\documentclass{scrartcl}
\usepackage{amssymb,amsmath,amsthm}
\usepackage{hyperref,comment}
\usepackage{mathrsfs}
\usepackage{caption}
\usepackage{graphicx, float, subfig}
\usepackage{color, mathtools, tikz, xcolor}
\usepackage{enumerate,enumitem}
\usepackage[capitalise]{cleveref}
\usepackage[noadjust]{cite}
\usepackage[margin=1.1in]{geometry}

\hypersetup{colorlinks=true,
   citecolor=blue,
   filecolor=blue,
   linkcolor=blue,
   urlcolor=blue}

\usepackage{pgfplots}
\pgfplotsset{width=10cm,compat=1.9}

\usepgfplotslibrary{fillbetween}

\numberwithin{equation}{section}
\creflabelformat{equation}{(#2#1#3)}
\crefdefaultlabelformat{#2#1#3}
\Crefname{enumi}{}{}

\newcommand{\COMMENT}[1]{}

	\newtheorem{theorem}{Theorem}[section]

	\newtheorem{corollary}[theorem]{Corollary}
	
	\newtheorem{question}[theorem]{Question}
	\newtheorem{claim}[theorem]{Claim}
	\newtheorem{proposition}[theorem]{Proposition}
	\newtheorem{lemma}[theorem]{Lemma}
	\newtheorem{fact}[theorem]{Fact}
	\newtheorem{definition}[theorem]{Definition}

\newenvironment{proofclaim}[1][Proof of claim]{\begin{proof}[#1]}{\end{proof}}

\def\eps{{\varepsilon}}

\title{Compatible Hamilton cycles in graphs with large minimum degree}

\author{Natalie Behague\thanks{School of Mathematical Sciences, Dublin City University, Dublin D09 W6Y4, Ireland; e-mail: \texttt{natalie.behague@dcu.ie}.
} \and Francesco Di Braccio\thanks{Department of Mathematics, London School of Economics and Political Science, London WC2A 2AE, United Kingdom; e-mail: \texttt{f.di-braccio@lse.ac.uk}.} \and Bertille Granet\thanks{School of Mathematics, Monash University, Melbourne VIC 3800, Australia; e-mail: \texttt{bertille.granet@monash.edu}.} \and Allan Lo\thanks{School of Mathematics, University of Birmingham, Birmingham B15 2TT, United Kingdom; e-mail: \texttt{s.a.lo@bham.ac.uk}.}}

\date{}

\begin{document}

\maketitle

\begin{abstract}
The renowned theorem of Dirac states that if $G$ is a graph with minimum degree at least $n/2$ then $G$ has a Hamilton cycle. A natural generalisation asks what properties of an edge-colouring of $G$ guarantee the existence of a properly edge-coloured Hamilton cycle in $G$. This concept can be further generalised as follows: an \emph{incompatibility system} for $G$ is a set~$\mathcal{F}$ of `forbidden' pairs of adjacent edges, that is, $\mathcal{F}\subseteq \{\{uv,vw\}\in \binom{E(G)}2\}$. A cycle in $G$ is then \emph{compatible} if no two of its edges form a pair in $\mathcal{F}$.
The system $\mathcal{F}$ is called \emph{$\mu n$-bounded} if for all $v\in V(G)$ and $uv\in E(G)$, there are at most $\mu n$ pairs $\{uv,vw\}\in \mathcal{F}$.
How small must $\mu$ be to guarantee the existence of a compatible Hamilton cycle in $G$?
Krivelevich, Lee and Sudakov showed that $\mu=10^{-16}$ suffices (for $n$ large), while an example of Bollob\'as and Erd\H{o}s shows that $\mu\leq 1/4$ is necessary. We significantly reduce this gap for large graphs of minimum degree at least $(1/2+\varepsilon)n$, by showing that $\mu=1/8$ suffices but $\mu\leq 1/6$ is necessary for such graphs. In fact, we give more precise bounds which are functions of $\delta(G)/n$.
\end{abstract}

\section{Introduction}

\subsection{Background and results}

Let $G$ be a graph on $n \ge 3$ vertices with $\delta(G) \ge n/2$.
In 1952, Dirac famously proved that $G$ contains a Hamilton cycle. Since then, a well-studied strand of research consists in generalising this seminal result, in particular by finding Hamilton cycles satisfying specific properties.

A typical example of this is to colour the edges of $G$ and seek Hamilton cycles with specific colour patterns, such as \emph{rainbow} Hamilton cycles, where all edges have distinct colours.
For example, Albert, Frieze and Reed~\cite{AFR} showed that every edge-coloured~$K_n$ with no colour appearing more than $n/32$ times contains a rainbow Hamilton cycle. 
This was then strengthened by Coulson and Perarnau~\cite{CoulsonPerarnau}, who showed that if an edge-coloured graph $G$ with $\delta(G) \ge n/2$  has no colour appearing more than $\mu n$ times, for some sufficiently small $\mu>0$, then it contains a rainbow Hamilton cycle.

In this paper, we focus on local restrictions on Hamilton cycles. 
Typically one seeks \emph{properly} edge-coloured Hamilton cycles, where any two incident edges receive distinct colours, in edge-coloured (complete) graphs. Let $\Delta_{\mathrm{mono}}(G)$ be the maximum number of times a colour appears at a vertex. 
Equivalently, $\Delta_{\mathrm{mono}}(G) \coloneqq \max \{ \Delta(H) : H \subseteq G \text{ is monochromatic}\}$. 
The Bollab\'as--Erd\H{o}s conjecture~\cite{BollobasErdos} states that $\Delta_{\mathrm{mono}}(K_n) < \lfloor n/2 \rfloor$ suffices for the existence of a properly edge-coloured Hamilton cycle. 
Following various results~\cite{BollobasErdos, shearer, alongutin, chendaykin} toward this conjecture, it was then solved asymptotically by the last author~\cite{AsymptoticBollobasErdos}.

Here, we will generalise proper edge-colourings further using incompatibility systems.
Let $G$ be a graph.  
An \emph{incompatibility system}~$\mathcal{F}$ for~$G$ is a family of pairs of intersecting edges of~$G$, that is, $\mathcal{F} \subseteq \{\{e,e'\} \in E(G) \colon e \cap e' \ne \emptyset\}$.
Here, $\mathcal{F}$ is to be understood as the list of `forbidden' pairs of edges. 
We say that a subgraph $H$ in~$G$ is \emph{$\mathcal{F}$-compatible} if there do not exist $e,e' \in E(H)$ such that $\{e,e'\} \in \mathcal{F}$, that is, $\binom{E(H)}{2} \cap \mathcal{F} = \emptyset$. 
For example, consider an edge-coloured graph~$G$.
If $\mathcal{F}_{\textrm{mono}}$ consists of all pairs of intersecting edges of the same colour, then $\mathcal{F}_{\textrm{mono}}$-compatible is equivalent to properly edge-coloured.

Recall that  for all $\{e,e'\} \in \mathcal{F}$, we have $e \cap e' \ne \emptyset$. 
Thus $\mathcal{F}$ can be written as $\bigcup_{v \in V(G)} \mathcal{F}_v$, where $\mathcal{F}_v$ is the list of forbidden pairs of edges containing the vertex~$v$, that is, $\mathcal{F}_v \subseteq \{ \{e,e'\} : e \ne e' \in E(G), e \cap e' = \{v\} \}$.
We say an incompatibility system~$\mathcal{F}$ is \emph{$\Delta$-bounded} if for each $v \in V(G)$ and edge $e$ containing~$v$, there are at most $\Delta$ other edges $e'$ such that $\{e,e'\}\in \mathcal{F}_v$.
Hence $\mathcal{F}_{\textrm{mono}}$ is $\Delta_{\textrm{mono}}(G)$-bounded.

An additional motivation for the study of compatible Hamilton cycles in Dirac graphs is that they also generalise \emph{transition systems}, which are  equivalent to 1-bounded incompatibility systems. Motivated by a problem of Nash-Williams on cycle coverings of Eulerian graphs, Kotzig \cite{Kotzig} first introduced this concept in 1968.
H\"{a}ggkvist (see \cite[Conjecture 8.40]{bondy}) then
conjectured that if $G$ is a Dirac graph and $\mathcal{F}$ is a 1-bounded incompatibility system (a.k.a.~transition system) for $G$,
then $G$ contains an $\mathcal{F}$-compatible Hamilton cycle. 
The first major result on this topic is due to Krivelevich, Lee and Sudakov~\cite{KLS}, who proved H\"{a}ggkvist's conjecture in a strong sense.

\begin{theorem}[Krivelevich, Lee and Sudakov~\cite{KLS}]\label{thm:KLS}
For $n$ sufficiently large, every $n$-vertex graph~$G$ with $\delta(G) \ge n/2$ with a $10^{-16} n$-bounded incompatibility system~$\mathcal{F}$ has an $\mathcal{F}$-compatible Hamilton cycle.
\end{theorem}
 
They also mentioned that the constant $10^{-16}$ has not been optimized.
This gives rise to the following natural question, also included in a survey of Sudakov~\cite{sudakov2017robustness}.
\begin{question}[Krivelevich, Lee and Sudakov~\cite{KLS,sudakov2017robustness}]\label{q:main}
What is the largest constant $\mu$ such that every $n$-vertex graph~$G$ with $\delta(G) \ge n/2 $ and a $\mu n$-bounded incompatibility system~$\mathcal{F}$ has an $\mathcal{F}$-compatible Hamilton cycle?
\end{question} 

We know that $10^{-16} < \mu \le 1/4$, where the lower bound is due to~\cref{thm:KLS}, and the upper bound is due to the following simple example based on a construction of Bollob\'{a}s and Erd\H{o}s~\cite{BollobasErdos}, see~\cite{KLS}.
Consider a red-blue coloured $n$-vertex graph~$G$ such that $n = 8k-1$ and $G$ is the union of two edge-disjoint spanning $2k$-regular graphs with one of each colour.
Since $G$ is $2$-edge-coloured, $G$ contains no properly coloured odd cycle, and so no properly coloured Hamilton cycle.
By defining $\mathcal{F}_{\textrm{mono}}$ as above, we deduce that $\mathcal{F}_{\textrm{mono}}$ is $(2k-1)$-bounded
and $G$ does not have an $\mathcal{F}_{\textrm{mono}}$-compatible Hamilton cycle.

In this paper, we consider $n$-vertex graphs~$G$ with $\delta(G) \ge (1/2+ \eps) n$ and improve the bounds on $\mu$ significantly to $ 1/8 \le \mu \le 1/6$. 
In fact, our bounds on $\mu$ are functions of $\delta(G)/n$. To that end, we introduce the following definition.
\begin{definition}\label{def:mu}
    For $\delta \ge 1/2$, let $\mu(\delta)$ be the supremum over all $\mu$ for which the following holds: for all $\rho>0$, there exists $n_0 = n_0(\delta, \rho)$ such that every $n$-vertex graph~$G$ with $n \ge n_0$, $\delta(G) \ge\delta n $ and a $(1- \rho)\mu n$-bounded incompatibility system has a compatible Hamilton cycle. 
\end{definition}

When $n$ is odd, a straightforward generalisation of the aforementioned example of Bollob\'as and Erd\H{o}s shows that $\mu(\delta) \le \delta/2$.
We provide another simple construction which covers all values of $n$, which can also be adapted to the corresponding problem for properly coloured Hamilton cycles.

\begin{proposition}\label{prop:upper_bound_reg}
Let $d\in \mathbb{N}$. 
Then for any $d$-regular graph~$G$, there exists a $\lceil d/2 \rceil$-bounded incompatibility system~$\mathcal{F}$ such that $G$ has no $\mathcal{F}$-compatible Hamilton cycle.
Moreover, there exists an edge-colouring of $G$ such that $\Delta_{\textrm{mono}}(G) \le \lceil d/2 \rceil+1$ without a properly coloured Hamilton cycle. 
\end{proposition}

Perhaps surprisingly, we give a new construction showing that $\mu(\delta) \le \delta -1/3$. 
Again, there is an edge-coloured analogue of this construction.

\begin{proposition}\label{prop:upper_bound_nonreg}
    Let $n \in \mathbb{N}$ and $1/2 \le \delta < 1$.
    There exists an $n$-vertex graph~$G$ with minimum degree $\delta(G) \ge \delta n$ and a $(\delta n -n/3 + 1)$-bounded incompatibility system~$\mathcal{F}$ such that $G$ has no $\mathcal{F}$-compatible Hamilton cycle.
    Moreover, there exists an edge-coloured $n$-vertex graph~$G$ with minimum degree $\delta(G) \ge \delta n$ and  $\Delta_{\textrm{mono}}(G) \le \delta n -n/3 + 2$ without a properly coloured Hamilton cycle. 
\end{proposition}

The main result of this paper is a lower bound on $\mu(\delta)$. 

\begin{theorem}\label{thm:main}
     For any $\rho> 0$ and $\delta >1/2$, there exists $n_0\in \mathbb{N}$ such that the following holds.
Let $G$ be an $n$-vertex graph with $n\geq n_0$ and $\delta(G) \ge \delta n$. 
If $\mathcal{F}$ is a $\mu n$-bounded incompatibility system for $G$ with
\begin{align*}
\mu\coloneqq (1-\rho)\frac{\delta + \sqrt{2\delta - 1}}{4},
\end{align*}
then $G$ contains an $\mathcal{F}$-compatible Hamilton cycle.
\end{theorem}

In summary, for $1/2<\delta<1$ we have 
\begin{align*} 
\frac{\delta + \sqrt{2\delta - 1}}{4} \le 
\mu (\delta) \le \min\left\{\frac{\delta}{2},~\delta - \frac{1}{3}\right\}.
\end{align*}
This bound is sharp when $\delta$ tends to~$1$ and more surprisingly, when $\delta =5/9$. See Figure~\ref{fig:feasible_mu} for diagram of the a feasible region for~$\mu(\delta)$.

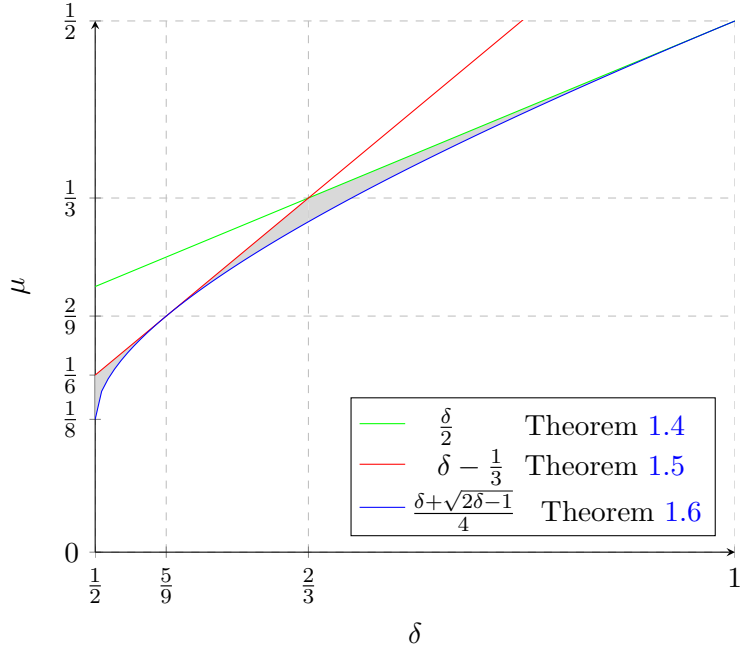
\begin{figure}[t]
    \centering
\begin{tikzpicture}
\begin{axis}[
    axis lines = left,
    xlabel = \(\delta\),
    ylabel = {\(\mu\)},
    xmin=1/2, xmax=1,
    ymin=0, ymax=.5,
    xtick={1/2,5/9,2/3,1},
xticklabels={$\frac{1}{2}$, $\frac{5}{9}$, $\frac{2}{3}$, 1},
    ytick={0,2/9,1/3,.5},
    yticklabels={0,$\frac{2}{9}$,$\frac{1}{3}$,$\frac{1}{2}$},
    extra y ticks={1/8, 1/6},
extra y tick style={grid=none},
extra y tick labels={$\frac{1}{8}$,$\frac{1}{6}$},
    ymajorgrids=true,
    xmajorgrids=true,
    grid style=dashed,
    legend pos=south east,
]
\addplot [ name path = red,
    domain=1/2:1, 
    samples=100, 
    color=green,
]
{x/2};
\addlegendentry{~$\frac{\delta}{2}$ \hspace{.655cm} \cref{prop:upper_bound_reg}}
\addplot [ name path = green,
    domain=1/2:1, 
    samples=100, 
    color=red,
    ]
    {x-1/3};
\addlegendentry{~$\delta - \frac{1}{3}$ \hspace{0cm} \cref{prop:upper_bound_nonreg}}
\addplot [ name path = blue,
    domain=.5:1, 
    samples=100, 
    color=blue,
    ]
    {(x+sqrt(2*x-1))/4};
\addlegendentry{$\frac{\delta + \sqrt{2\delta - 1}}{4}$ \hspace{.07cm} \cref{thm:main}}
\addplot[black!15] fill between[of=blue and green,
 soft clip = {domain=.5:2/3}, 
];
\addplot[black!15] fill between[of=blue and red,
 soft clip = {domain=2/3:1}, 
];
\end{axis}
\end{tikzpicture}
    \caption{The bounds on $\mu\coloneqq \mu(\delta)$. The feasible region is shaded.}
    \label{fig:feasible_mu}
\end{figure}

The lower bound on $\mu(\delta)$ arises from the fact that in our proof of~\cref{thm:main}, we would need to find a spanning $\delta' n$-regular subgraph of~$G$ with $2 \mu < \delta'$.
By Tutte's $f$-factor theorem~\cite{Tutte} (or see discussion in Section~\ref{sec:regularising}),  every $n$-vertex graph with $\delta(G) = \delta n \ge n/2$ contains a $2 \ell$-regular spanning subgraph with $\ell = \left\lfloor (\delta + \sqrt{2\delta - 1})n/4\right\rfloor$.
Therefore, it is of no surprise that we also prove a regular graph version of~\cref{thm:main}. In fact, we prove a more general statement that holds for almost regular graphs. \cref{prop:upper_bound_reg} implies that this theorem is asymptotically sharp. 

\begin{theorem}[Almost regular graphs]\label{thm:reg}
    For any $\rho,\eps > 0$, there exist $\gamma>0$ and $n_0\in \mathbb{N}$ such that the following holds.
		Let $G$ be an $n$-vertex graphs such that $n \ge n_0$ and $(d - \gamma) n \le \delta (G) \le \Delta(G) \le (d+\gamma) n$ for some $ d \ge 1/2+\eps$. 
		If $\mathcal{F}$ is a $\mu n$-bounded incompatibility system for~$G$ with $\mu\coloneqq (1-\rho)d/2$, then $G$ contains an $\mathcal{F}$-compatible Hamilton cycle.
\end{theorem}

Structures other than Hamilton cycles have also been considered with respect to incompatibility systems, including powers of Hamilton cycles~\cite{cheng2025compatible} and graph tilings~\cite{hu2023graph}. 
However, as in~\cref{thm:KLS}, the lower bounds on~$\mu$ obtained in these papers are very small and likely to be far from the truth.

One can also generalise compatible systems further as follows, to allow pairs of edges in~$\mathcal{F}$ that may not be intersecting, that is, $\mathcal{F} \subseteq \binom{E(G)}{2}$.
If $G$ is an edge-coloured graph and $\mathcal{F}_{\textrm{mono}}'$ consists of all pairs of edges of the same colour, then $\mathcal{F}_{\textrm{mono}}'$-compatible is equivalent to rainbow. 

Furthermore, we can even let $\mathcal{F} \subseteq 2^{E(G)}$ be a family of all subsets of~$E(G)$ and seek a subgraph of~$G$ that does not contain any members of~$\mathcal{F}$. 
In the literature, these are sometimes known as \emph{configuration graphs} and \emph{$\mathcal{F}$-avoiding subgraphs} respectively, see for example the work of Delcourt and Postle \cite{delcourtpostleconflictfree}, and sometimes known as \emph{conflict systems} and \emph{conflict-free subgraphs} respectively, see for example the work of Glock, Joos, Kim, K\"uhn, and Lichev \cite{conflictfree}.

\subsection{Corollaries for properly coloured Hamilton cycles}

As previously discussed, compatible Hamilton cycles generalise properly coloured Hamilton cycles, through the incompatibility system $\mathcal{F}_{\textrm{mono}}$ consisting of all pairs of intersecting edges of the same colour. 
Recall that $\Delta_{\mathrm{mono}}(G) \coloneqq \max \{ \Delta(H) : H \subseteq G \text{ is monochromatic}\}$. 
Although they are straightforward corollaries, we explicitly state the following two results on properly coloured Hamilton cycles as they are likely to be of independent interest.
\begin{corollary}\label{cor:properly_coloured_main}
     For any $\rho> 0$ and $\delta>1/2$, there exists $n_0\in \mathbb{N}$ such that the following holds.
     If $G$ is an edge-coloured $n$-vertex graph with $n\geq n_0$, 
\begin{align*}
\delta(G) \ge \delta n
\text{ and }
\Delta_{\mathrm{mono}}(G) \le  (1-\rho)\frac{\delta + \sqrt{2\delta - 1}}{4}n,
\end{align*}
then $G$ contains a properly coloured Hamilton cycle.
\end{corollary}

\begin{corollary}\label{cor:properly_coloured_reg}
    For any $\rho,\eps > 0$, there exist $\gamma > 0 $ and $n_0\in \mathbb{N}$ such that the following holds.
    Let $G$ be an $n$-vertex graph where $n\geq n_0$ and $(d-\gamma)n \le \delta(G) \le \Delta(G) \le (d+\gamma)n$ for some $d \ge 1/2 + \eps$. If $G$ is edge-coloured such that $\Delta_{\mathrm{mono}}(G) \leq (1-\rho)dn/2$, 
then $G$ contains a properly coloured Hamilton cycle.
\end{corollary}

The two corollaries  are immediate from \cref{thm:main} and \cref{thm:reg} respectively, using the incompatibility system $\mathcal{F}_{\textrm{mono}}$ defined above.
Corollary~\ref{cor:properly_coloured_reg} is asymptotically sharp by Proposition~\ref{prop:upper_bound_reg}.
Note that by taking $d = 1$, Corollary~\ref{cor:properly_coloured_main} implies that if $\delta(G) \ge (1-\gamma)n$ and $\Delta_{\mathrm{mono}}(G) \le (1-\rho) n/2$, then there is a properly coloured Hamilton cycle.
Thus, we recover an asymptotic version of the Bollab\'as--Erd\H{o}s conjecture, which is also generalised to almost complete graphs.
(The original proof in~\cite{AsymptoticBollobasErdos} relies heavily on the fact that $G$ is complete, so it does not generalise to almost complete graphs easily.)

\subsection{Paper outline and proof overview}

We start with the two constructions proving~\cref{prop:upper_bound_reg,prop:upper_bound_nonreg} in \cref{sec:upper_bounds} and some preliminaries in \cref{sec:preliminaries}. We end with some open problems in \cref{sec:open_problems}. The rest of the paper is devoted to the proofs of~\cref{thm:reg,thm:main}, which follow a common strategy which we now briefly present.

One of the most famous techniques for finding long paths is P\'{o}sa's rotation-extension.
Roughly speaking, given a path~$P= v_1 \dots v_{\ell}$ of maximum length and an edge~$v_1 v_i$ with $i \ge 3$, we have that either $P \cup v_1v_i$ is a cycle or $(P \cup v_1v_i) - v_{i-1}v_i$ is another path of the same length.
The latter is often referred to as a \emph{rotation of~$P$}. 
However, as observed in~\cite{KLS}, a compatible path may not have a single rotation. 
One idea that was used to circumvent this issue in the proof of \cref{thm:KLS} in~\cite{KLS} is to restrict to compatible paths with stronger properties that enable rotations. 
Unfortunately, this results in a very small~$\mu$. 

In this paper, we instead adapt another celebrated technique for constructing Hamilton cycles: absorption, which is often credited to R\"odl, Ruci\'nski and Szemer\'edi~\cite{rodl2008approximate}.
The proof is divided into three key lemmas: connecting, absorbing and covering.

The connecting lemma enables us join any two compatible paths together using constantly many vertices. 
The absorbing lemma states that one can reserve a small set of vertices~$A$ such that given a compatible linear forest~$H$ in $G\setminus A$ with at most $\beta n$ component (that is, $H$ is union of at most $\beta n$ paths), there exists a compatible cycle covering~$V(H) \cup A$. 
Finally, the covering lemma states that we can indeed find such a compatible linear forest.

Adapting this methodology to the compatible setting is non-trivial.
The main difficulty lies in the covering lemma, that is, finding an almost spanning compatible linear forest. 
Fix an equipartition of $V(G)$ into $V_1, \dots, V_m$, each of size $\beta n$ and with $m$ even. 
A linear forest can be constructed by picking a perfect matching~$M_i$ between $V_i$ and $V_{i-1}$ for each $i <m$. 
But clearly, an arbitrary choice of such a linear forest may not be compatible. 
A key property is that the incompatible system~$\mathcal{F}$ only imposes local restrictions, that is, $e \cap e' \ne \emptyset$ for all $\{e,e'\} \in \mathcal{F}$.
Consider the union of the odd matchings, $M_1 \cup M_3 \cup \dots \cup M_{m-1}$, which is itself a matching and so is compatible.
The choices of $M_2, M_4, \dots, M_{m-2}$ are now independent of each other. 
Therefore if we know how to pick $M_2$, then we should know how to pick the rest and so obtain a compatible linear forest. We discuss this strategy in more detail in \cref{sec:linear_forest}.

\subsection{Notation}

Throughout the paper, we use the following standard notation. All graphs in this paper are \emph{simple}, that is, without loops and multiple edges. Let $G$ be a graph. We denote by $V(G)$ and $E(G)$ the \emph{vertex set} and \emph{edge set} of $G$, respectively. An edge between two vertices $u$ and $v$ is denoted by $uv$ or $vu$ interchangeably. The \emph{size} of $G$ is $e(G) \coloneqq |E(G)|$. Given a set of vertices $X \subseteq V(G)$, we denote by $G[X]$ the \emph{induced subgraph} of $G$ on $X$ and define $G-X\coloneqq G[V(G)\setminus X]$. Similarly, if $\mathcal{F}$ is an incompatibility system for $G$ and $X\subseteq V(G)$, then we denote by $\mathcal{F}[X]$ the incompatibility system for $G[X]$ \emph{induced} by $X$, that is, $\mathcal{F}[X]\coloneqq \{\{e,e'\}\in \mathcal{F}: V(e)\cup V(e')\subseteq X\}$, and we write $\mathcal{F}-X\coloneqq \mathcal{F}[V(G)\setminus X]$.

Given a vertex $v \in V(G)$, we denote by $N_G(v)$ the \emph{neighbourhood} of $v$ in $G$, and by $d_G(v) \coloneqq |N_G(v)|$ the \emph{degree} of $v$ in $G$. Given a set of vertices $X \subseteq V(G)$, we let $d_G(v, X) \coloneqq |N_G(v) \cap X|$. We denote by $\delta(G) \coloneqq \min_{v \in V(G)} d_G(v)$ the \emph{minimum degree} of $G$, and by $\Delta(G) \coloneqq \max_{v \in V(G)} d_G(v)$ its \emph{maximum degree}.

A \emph{directed graph}, or \emph{digraph}, is a graph in which edges have a direction and up to two edges between any pair of vertices are allowed, one for each direction. An \emph{oriented graph} is a digraph $D$ with at most one edge between a pair of vertices, in other words, $D$ can be obtained by orienting the edges of an undirected graph.
A directed edge from a vertex $u$ to a vertex $v$ is denoted by $uv$.
Given a digraph $D$, its vertex set is denoted $V(D)$ and its edge set $E(D)$. We use $e(D) \coloneqq |E(D)|$ to denote its \emph{size}. 
Given vertex sets $X, Y \subseteq V(D)$, we use $D[X,Y]$ to denote the subgraph of $D$ consisting of all the edges $xy\in E(D)$ such that $x\in X$ and $y\in Y$, and denote by $e_D(X,Y) \coloneqq |E(D[X,Y])|$ the size of $D[X,Y]$. Given a vertex $v \in V(D)$, we denote by $N^+_D(v)$  the \emph{outneighbourhood} of $v$ and by $N^-_D(v)$ its \emph{inneighbourhood}. Also, $d^+_D(v) \coloneqq |N_D^+(v)|$ denotes the \emph{outdegree} of $v$ and $d^-_D(v) \coloneqq |N_D^-(v)|$ its \emph{indegree}; additionally, $d^+_D(v, X) \coloneqq |N^+_D(v) \cap X|$ and $d^-_D(v, X) \coloneqq N^-_D(v, X)$ denote its in and outdegree into the set $X$, respectively.
For all the above, we omit the subscript $G$ or $D$ if it is clear from context.

The \emph{length} of a path is the number of its edges. The \emph{endpoints} of a path $v_0v_1\dots v_\ell$ are $v_0$ and $v_\ell$. When we say $xyQzw$ is path, we implicitly mean that $Q$ is a path with $y$ and $z$ as endpoints. 

In our statements, we will often make assumptions of the form $0 < a_1 \ll \dots \ll a_n$. This should be interpreted as saying that there exist positive real-valued functions $f_1, \dots, f_{n-1}$ such that the statement holds for any choice of $a_1, \dots, a_n > 0$ with $a_{i} \leq f_i(a_{i+1})$ for each $i \in [n-1]$. When writing $1/a$ in place of some $a_i$, we further implicitly assume that $a \in \mathbb{N}$. 

Given $a, b, c, d \in \mathbb{R}$, the notation $a \pm b = c \pm d$ should be interpreted as $[a - b, a + b] \subseteq [c - d, c+d]$. Throughout the paper, we omit floors and ceilings whenever they are not critical to the argument.

\section{Upper bounds}\label{sec:upper_bounds}

In this section, we first prove~\cref{prop:upper_bound_reg}, that is, showing $\mu(\delta) \le \delta/2$. 
To cover the case where $n$ is even, we generalise a different extremal example for the Bollob\'as--Erd\H{o}s conjecture that was first observed by Fujita and Magnant~\cite{FujitaMagnant}.

\begin{proof}[Proof of~\cref{prop:upper_bound_reg}]
Let $G$ be any $d$-regular graph. 
Clearly, we may assume that $G$ contains a Hamilton cycle as otherwise the result is trivial. 

Orient $G$ so that $d^-(v), d^+(v) \in \{ \lfloor d/2\rfloor, \lceil d/2 \rceil\}$ for all~$v \in V(G)$. 
This is possible since, when $d$ is even, we can orient along an Eulerian walk on~$G$, and when $d$ is odd (and so $|V(G)|$ is even), there is a perfect matching~$M$ in~$G$ (as $G$ contains a Hamilton cycle) and so we can orient $G - M$ into a regular oriented graph and use an arbitrary orientation on~$M$. 

Fix a vertex $x \in V(G)$ and re-orient all edges containing~$x$ to be directed away from~$x$. 
Call the resulting oriented graph~$H$. 
Set $\mathcal{F}$ to be the set of $\{uv,uw\} \in \binom{E(G)}2$ such that $u \ne x$ and either $uv, uw \in E(H)$ or $vu, wu \in E(H)$. 
Note that $\mathcal{F}$ is $\lceil d/2 \rceil$-bounded and that an $\mathcal{F}$-compatible cycle in~$G$ is a directed cycle in~$H$. 
However $H$ does not have a directed Hamilton cycle as $x$ is a source, as desired. 

For the `moreover part', define an edge-colouring on~$G$ with colour set $V(G)$ such that the edge $uv \in E(G)$ receives colour~$v$ if and only if $uv \in E(H)$.  
Note that a properly coloured cycle corresponds to a directed cycle in~$H$. 
Thus it is the desired edge-colouring of~$G$.
\end{proof}

Next we exhibit a different construction to prove~\cref{prop:upper_bound_nonreg}, which significantly improves this upper bound.

\begin{proof}[Proof of~\cref{prop:upper_bound_nonreg}]
    Let $a$ be the unique even integer such that $2n/3 < a \le 2n/3 + 2$.
    Let $A$ and $B$ be disjoint vertex sets of sizes $a$ and $n-a$ respectively. Let $G$ be a graph on $A\cup B$ where  each vertex in $B$ joined to all other vertices and $G[A]$ is $(\delta n -(n-a))$-regular. 
    Note that since $a$ is even and $a > (\delta n -(n-a)) > 0$ it is always possible to find a $(\delta n -(n-a))$-regular graph $G[A]$  -- for example, take an appropriate circulant graph.
    Let \[\mathcal{F} \coloneqq \left\{ \{e,e'\} \in \binom{E(G[A]))}{2} \colon e \cap e' \ne \emptyset\right\}.\]
    Note that this is the incompatibility system $\mathcal{F}_{\textrm{mono}}$ for the edge colouring of $G$ where all edges in $A$ are coloured red and all other edges receive unique colours. 
    The resulting incompatibility system $\mathcal{F}$ is $(\delta n-n +a - 1)$-bounded, where we have $\delta n -n+a - 1\leq \delta n-n/3+1$ since $a \le 2n/3 + 2$.

    Since $a> 2n/3$, we have that $|A|>2|B|$. Therefore, any Hamilton cycle in $G$ contains two consecutive edges in $A$. However, any two such edges are incompatible, as desired.  

    For the `moreover part', we now edge-colour $G$ such that all edges in $G[A]$ use red and each remaining edge of~$G$ receives a unique colour. 
    A properly coloured cycle is $\mathcal{F}$-compatible, so $G$ has no properly coloured Hamilton cycle. 
\end{proof}

\section{Preliminaries}\label{sec:preliminaries}

\subsection{Probabilistic estimates}

Let $n,m, k \in \mathbb{N}$ with $m,k < n$.
Recall that the hypergeometric distribution with parameters $n, m$ and $k$ is the distribution of the random variable $X$ defined as follows.
Let $M$ be a random subset of $[n]$ of size $m$ and let $X \coloneqq |M \cap [k]|$.

We will need the following standard Chernoff bound, as well as McDiarmid's and Markov's inequalities.

\begin{lemma}[{Chernoff's inequality (see e.g.~\cite[Corollary~2.3 and Theorem~2.10]{JLR})}]\label{lm:chernoff}
    Let $X$ denote the sum of independent identically distributed binomial random variables, or a hypergeometric random variable. Let $\mu \coloneqq \mathbb{E}[X]$. Then for any $0 \le \delta \le 1$, we have
    \[
    \mathbb{P}(|X- \mu| \geq \delta \mu) \le 2 \exp(-\delta^2\mu/3).
    \]
\end{lemma}

For simplicity of use, we state the following corollary of \cref{lm:chernoff}, which follows by applying it to the hypergeometric distribution with parameters $n, m, (\alpha + \lambda)n$.

\begin{corollary}\label{cor:chernoff2}
Let $\alpha , \lambda >0$ with $\alpha + \lambda <1$. 
Suppose that $Y \subseteq [n]$ and $ |Y| \ge (\alpha + \lambda)n$. 
Then 
\begin{align*}
\left| \left\{ M \in \binom{[n]}{m} : |M \cap Y | \le \alpha m \right\} \right| \le 2 \binom{n}m e^{-\lambda^2 m /3}.
\end{align*}
\end{corollary}

\begin{theorem}[McDiarmid's inequality \cite{mcdiarmid1989method}]\label{thm:mcdiarmid} Let $n \in \mathbb{N}$ and $c_1, \dots, c_n \geq 0$. For each $i \in [n]$, let $X_i$ be an independent random variable taking values in $\Omega_i$, and let $X \coloneqq (X_1, \dots, X_n)$. Let $f: \prod_{i=1}^m \Omega_i \to \mathbb{R}$ be a function such that, for each $i \in [n]$, changing $X$ in the $i$th coordinate changes the value of $f(X)$ by at most $c_i$. Then, for all $t >0$, 
\[\mathbb{P}\left(\left|f(X) - \mathbb{E}[f(X)] \right| \geq t \right) \leq 2\exp \left(-\frac{t^2}{\sum_{i \in [n]}c_i^2}\right).\]
\end{theorem}

\begin{lemma}[{Markov's inequality (see \cite[Equation~1.3]{JLR})}]\label{lm:Markov}
    For any non-neqative random variable $X$ and $a>0$, we have 
    \[\mathbb{P}(X>a\mathbb{E}[X])\leq 1/a.\]
\end{lemma}

\subsection{Regularising}\label{sec:regularising}

When proving \Cref{thm:main}, 
it will be useful to first find a spanning regular subgraph and then to find a compatible linear forest in that regular subgraph. The following result allows us to pass to a regular subgraph.

\begin{theorem}[Christofides,  K\"{u}hn and Osthus {\cite[Theorem 7]{regular_subgraph}}]\label{thm:regular_subgraph}
Let $1/n \ll \eps$, let $\delta \ge (1/2 + \eps)$ and let $\ell \in \mathbb{N}$ with $\ell \le (\delta - \eps + \sqrt{2\delta - 1})n/4$.
Then every $n$-vertex graph with minimum degree $\delta n$ contains a spanning connected $2\ell$-regular subgraph.
\end{theorem}

Note that~\cref{thm:regular_subgraph} is essentially tight up to the $\eps$ factor: for all positive integers~$n$ with and all $1/2 < \delta < 1$ such that  $\delta n$ is an integer, Christofides,  K\"uhn and Osthus~\cite{regular_subgraph} exhibit an $n$-vertex graph~$G$ with minimum degree $\delta(G) = \delta n$ such that~$G$ contains no $r$-regular graph for any 
$r > (\delta  + \sqrt{(2\delta - 1)})n/2 + 1$.

\bigskip

There are better bounds when the graph is almost regular.

\begin{theorem}[{\cite[Theorem 6]{regular_subgraph}}]\label{lm:regularising}
    Let $0<1/n\ll \varepsilon \ll 1$ and $\delta \ge 1/2+\eps$.
    Then every $n$-vertex graph with $\delta n \le d(v) \le (\delta  + \varepsilon^2 / {5})n$ for all vertices~$v$
    contains a spanning regular subgraph with degree at least $(\delta -\varepsilon) n$.
\end{theorem}

\Cref{thm:regular_subgraph} also enables us to `regularise' our incompatibility system.
More precisely, we say an incompatibility system~$\mathcal{F}$ is \emph{$d$-regular} if for every vertex $v \in V(G)$ and every edge $e$ containing $v$, there are exactly $d$ edges  $e'$ such that $\{e,e'\}\in \mathcal{F}$. 
The next proposition allows us to pass from a bounded incompatibility system to a regular one.

\begin{proposition}\label{prop:regF}
    Let  $0 < 1/n \ll \rho ,d \le 1$ and $\mu \coloneqq (1-\rho)d/2$.
    Define $\mu'$ such that $\mu'n = dn - 2\lfloor dn/8 \rfloor -1$.
Let $G$ be an $n$-vertex $d n $-regular  graph and let $\mathcal{F}$ be a $\mu n $-bounded  incompatibility system for $G$.
Then there is a $\mu'n$-regular incompatibility system $\mathcal{F}'$ for~$G$ containing~$\mathcal{F}$.
\end{proposition}

\begin{proof}
    Let $\eps \coloneqq \rho/4$ and $\delta \coloneqq 1/2 + \rho/4$, so that
    \begin{align*}
        \frac{\delta - \eps + \sqrt{2\delta - 1}}{4} \ge \frac{ 1/2 + \sqrt{{\rho}/{2} }}{4} > \frac18.
    \end{align*}
    Let $v\in V(G)$. Let $H_v$ be the graph on vertex set $N_G(v)$ such that $uw \in E(H_v)$ if and only if $\{ vu  ,vw \} \notin \mathcal{F}_v $.
    Since $G$ is $dn$-regular and $\mathcal{F}$ is $\mu n$-bounded, $H_v$ is a $dn$-vertex graph with minimum degree
    $\delta(H_v)  \ge dn -1 - \mu n = (1+\rho)dn/2 - 1 > \left(1/2 + \rho/4\right)dn$
    and so, by \cref{thm:regular_subgraph}, it has a spanning $2\lfloor dn/8 \rfloor$-regular subgraph~$H_v'$.
    
    Now define $\mathcal{F}'_v$ to be the set of all $\{ vu  ,vw \}$ such that $uw \notin E(H_v')$ and let $\mathcal{F}'\coloneqq \bigcup_{v\in V(G)}\mathcal{F}_v'$. Given $uv\in E(G)$, the number of edges $uv$ is compatible with is $d_{H_v'}(u)$ and so the number of edges it is incompatible with is $dn-1-d_{H_v'}(v)=dn - 2\lfloor dn/8 \rfloor -1 = \mu'n$. In particular, $\mathcal{F}'$ is a $\mu'n$-regular incompatibility system for $G$ containing $\mathcal{F}$, as required.
\end{proof}

\subsection{Matchings}

We will need two easy corollaries of K\"onig's theorem.

\begin{lemma}[{K\"onig's theorem \cite{konig1916graphok} (see also \cite{konig1916graphen} for a German translation)}]\label{lm:Koenig}
    Any bipartite graph $G$ can be decomposed into $\Delta(G)$ matchings.
\end{lemma}

 \begin{corollary}[{see e.g.~\cite[Exercise 7.1.33]{west2001introduction}}]\label{lm:matdecomp}
    For all $D\in \mathbb{N}$, any bipartite graph $G$ with $\Delta(G)\leq D$ can be decomposed into $D$ matchings $M_1, \dots, M_D$ such that $||M_i|-|M_j||\leq 1$ for all $i,j\in [D]$.
\end{corollary}

\begin{corollary}\label{lem:almost_matching}
    Any $2n$-vertex balanced bipartite graph $G$ such that $d_G(v)=d\pm r$ for all vertices $v\in V(G)$, with $d+r\leq n$,
    contains a matching with at least $n - 2rn/(d + r)$ edges.
\end{corollary}

\begin{proof}
    We have $e(G) \ge (d-r)n$.
    By \cref{lm:Koenig}, $G$ can be partitioned into $\Delta(G) \le d+r$ matchings. 
    Hence there exists a matching of size at least
    \begin{align*}
    \frac{e(G)}{\Delta(G)} \ge \frac{(d-r)n}{d+r} = \left( 1- \frac{2r}{d+r}\right) n = n - \frac{2rn}{d+r},
    \end{align*}
    as desired.
\end{proof}

\section{Finding a compatible linear forest}
\label{sec:linear_forest}

In this section we prove that the graph contains a compatible linear forest on many edges, that is, consisting of few paths.
We first give an overview of our approach.
 
A folklore method to construct a large linear forest with few paths in an $n$-vertex $dn$-regular graph~$G$ proceeds as follows. (Here we have $0< d <1$.)
First randomly partition $V(G)$ into equal-sized $S_1, \dots, S_k$ such that $G[S_i,S_{i+1}]$ is (almost) $d|S_i|$-regular for each $i\in [m-1]$. 
Hence for each $i\in [k-1]$,  $G[S_i,S_{i+1}]$ contains an (almost) perfect matching (by Corollary~\ref{lem:almost_matching}).
The union $\bigcup_{i\in[k-1]}M_i$ is then an (almost) spanning linear forest which consists of (roughly) $|S_1|$ paths.

We now attempt to apply this to incompatible systems.
Let $G$ be an $n$-vertex graph with $\delta(G) \geq (1/2+ \eps)n$ and a $\mu$-bounded incompatible system~$\mathcal{F}$. 
By Theorem~\ref{thm:regular_subgraph} we may now assume that $G$ is $dn$-regular for some $0<d<1$.

Let $S_1, \dots, S_k$ be as above.
A naive approach would be to choose the matchings one by one: first choose a matching~$M_1$ of~$G[S_1,S_2]$ freely, and then for each $i \ge 2$ in turn, choose a matching~$M_{i}$ in the subgraph $G_{i}^*$ of $G[S_{i},S_{i+1}]$ induced by the edges which are compatible with~$M_i$. 
There is however no guarantee that $G_{i}^*$ is anywhere close to being regular, and so $M_{i}$ could be very small.
We now highlight two main issues of this approach that prevent~$G^*_i$ from being almost regular and explain how we address them.  

Firstly, the subgraph $G_{i}^*$ is obtained from~$G[S_{i},S_{i+1}]$ by deleting all edges that are incompatible with~$M_{i-1}$.
This edge deletion is not symmetrical as it is based on how edges are intersecting at~$S_{i}$. 
We overcome this by choosing an almost perfect matching~$M_i$ in~$G[S_{i},S_{i+1}]$ for each \emph{odd} $i\in [k-1]$.
For each \emph{even} $i\in [k-1]$, $G^*_i$ is the induced subgraph of $G[S_i,S_{i+1}]$ given by the edges which are compatible with~$M_{i-1}$ and~$M_{i+1}$.

Consider even $i \in [k-1]$.
To ensure that $G^*_i$ is almost regular, it is natural to choose $M_{i-1}$ and $M_{i+1}$ randomly. 
Recall that $G[S_{i-1},S_{i}]$ is almost $d|S_1|$-regular, so an edge~$e'$ in~$G[S_{i-1},S_{i}]$ appears in $M_{i-1}$ with probability roughly~$d$ (and a similar statement for $e' \in G[S_{i+1},S_{i+2}]$).
However the probability that an edge~$uv$ lies in~$G[S_{i},S_{i+1}]$ may still vary dramatically depending on $\mathcal{F}_u$ and $\mathcal{F}_v$. 
Our second idea is to ensure that $\mathcal{F}$ is regular, which can be achieved by~\cref{prop:regF}.

For technical reasons, the random matchings~$M_i$ for odd $i\in [k-1]$ will not be sampled from the set of all large matchings of~$G[S_i,S_{i+1}]$.
We also randomly partition each~$S_j$ further into $S_j^{(1)}, \dots, S_{j}^{(\sqrt{k})}$.
The random matching~$M_i$ is now the union of an independently chosen random matching of $G[S_i^{(j)},S_{i+1}^{(j)}]$ for each~$j\in [\sqrt{k}]$.
This enables us to apply McDiarmid's inequality (\cref{thm:mcdiarmid}) to show that for even $i \in [k-1]$, $G^*_i$ is almost regular with high probability.

Therefore, our main goal of this section is to prove the following lemma. 

\begin{lemma}\label{lem:nearspanninglinearforest}
     Let $0 < 1/n \ll \mu \le 4d/5 \leq 4/5$.
     Let $G$ be an $n$-vertex $dn$-regular graph and let $\mathcal{F}$ be a $\mu n$-regular incompatibility system for $G$. Then $G$ contains an $\mathcal{F}$-compatible linear forest on at least $n - n^{32/33}$ edges. 
\end{lemma}

Before we prove this lemma, note that by~\cref{prop:regF}, the regularity condition on~$\mathcal{F}$ in \cref{lem:nearspanninglinearforest} can be relaxed to $\mu n$-boundedness.

\begin{corollary}\label{cor:nearspanninglinearforest}
    Let $0 < 1/n \ll \rho, d \leq 1$ and $\mu\coloneqq (1-\rho)d/2$.
    Let $G$ be an $n$-vertex $dn$-regular graph and let $\mathcal{F}$ be a $\mu n$-bounded incompatibility system for~$G$. 
    Then $G$ contains an $\mathcal{F}$-compatible linear forest on at least $n - n^{32/33}$ edges. 
\end{corollary}
\begin{proof}
    Define $\mu'$ such that $\mu'n= dn-2\lfloor dn/8\rfloor-1$. By \cref{prop:regF}, $G$ has a $\mu'n$-regular incompatibility system $\mathcal{F}'$ which contains $\mathcal{F}$. Clearly $\mu' \le 4d/5$ as $1/n \ll d$ and so we apply~\cref{lem:nearspanninglinearforest} to get the result.
\end{proof}

\begin{proof}[Proof of~\cref{lem:nearspanninglinearforest}]

For each edge $e \in E(G)$, vertex $v\in e$, and vertex subset $S \subseteq V(G)$, let 
\[d_{\mathcal{F}}(e, v, S) \coloneqq |\{u \in V(G)\in S : vu \in E(G) \mbox{ and } \{e, vu\} \in \mathcal{F}\}|,\]
be the number of incompatible edges (at $v$) that $e$ sends to $S$.
Let $m\in \{ \lfloor n^{1/4} \rfloor , \lfloor n^{1/4} \rfloor -1\}$ be even. 
Randomly partition $V(G)$ into $\lfloor n/m\rfloor\geq m^3$ sets of size $m$ and at most one set of size less than $m$.
Delete the set of size less than~$m$ and let $\{S_i^{(j)} : i \in [m^2],\, j \in [m]\}$ be a set of $m^3$ of those sets of size~$m$.
By Chernoff's bound for the hypergeometric distribution (\cref{lm:chernoff}), for each $v \in V(G)$, $i\in [m^2]$, and $j \in [m]$, we have
\COMMENT{$\mathbb{P}\left(d(v, S_i^{(j)}) \neq dm \pm m^{3/5} \right)\leq \mathbb{P}\left(d(v, S_i^{(j)}) \neq (1\pm m^{-2/5})\mathbb{E}[d(v, S_i^{(j)})]\right)\leq 2\exp(-\frac{dm^{1/5}}{3})\leq \exp(-\frac{dn^{1/20}}{100})$.}
\[\mathbb{P}\left(|d(v, S_i^{(j)}) - dm|> m^{3/5} \right) \leq e^{- n^{1/21}}.\]
Similarly, for each $e \in E(G)$, $v \in e$, $i \in [m^2]$ and $j \in [m]$, we have%
\COMMENT{$\mathbb{P}\left(d_{\mathcal{F}}(e, v, S_i^{(j)}) \neq \mu m\pm m^{3/5}\right) \leq \mathbb{P}[d_{\mathcal{F}}(e, v, S_i^{(j)}) \neq (1\pm m^{-2/5})\mathbb{E}[X]]\leq 2e^{-\mu m^{1/5}/3}\leq e^{-\frac{\mu n^{1/20}}{100}}$.}
\[\mathbb{P}\left(|d_{\mathcal{F}}(e, v, S_i^{(j)}) - \mu m|> m^{3/5}\right) \leq e^{-n^{1/21}}.\]
Then, with positive probability we have $d(v, S_i^{(j)}) = dm\pm m^{3/5}$ for each $v \in V(G)$, $i\in [m^2]$ and $j \in [m]$, as well as  $d_{\mathcal{F}}(e, v, S_i^{(j)}) = \mu m\pm m^{3/5}$ for each $e \in E(G)$, $v \in e$, $i \in [m^2]$ and $j \in [m]$. Fix a choice of $\{S_i^{(j)}\}_{i \in [m^2], j\in [m]}$ satisfying these conditions. For each $i \in [m^2]$, let $S_i \coloneqq \bigcup_{j \in [m]} S_i^{(j)}$.

Given $i \in [m^2-1]$ and $j, j' \in [m]$, note that $G_{i,j,j'}\coloneqq G[S_i^{(j)}, S_{i+1}^{(j')}]$ is a balanced bipartite graph with degrees $dm\pm m^{3/5}$.
Observe that for each choice of $j_1, j_2, j_3, j_4 \in [m]$, each edge $uv \in E(G_{i,j_2,j_3})$, where $u\in S_i^{(j_2)}$ and $v\in S_{i+1}^{(j_3)}$, is compatible with $ (d-\mu)m \pm 2m^{3/5}$ edges $vw \in E(G_{i+1,j_3,j_4})$ and with $ (d-\mu)m \pm 2m^{3/5}$ edges $tu \in E(G_{i-1,j_1,j_2})$.

Let $D\coloneqq dm+2m^{3/5}$. For each odd $i \in [m^2-1]$ and $j \in [m]$, let $\mathcal{M}_{i,j}=\{M_{i,j}^{(k)}: k\in [D]\}$ be the decomposition of $G_{i,j,j}$ obtained by applying \cref{lm:matdecomp} and note that
\begin{equation}\label{eq:M}
    |M_{i,j}^{(k)}|\geq \frac{(dm-m^{3/5})m}{D}-1\geq m-n^{1/5}
\end{equation}
for each $k\in [D]$. For each odd $i \in [m^2-1]$ and $j \in [m]$, let $M'_{i,j}$ be chosen uniformly at random from $\mathcal{M}_{i,j}$ and independently of other choices of $i$ and $j$. Note that $M'_i \coloneqq \bigcup_{j \in [m]} M'_{i,j}$ is a matching of $G[S_i, S_{i+1}]$.  For each even $i \in [m^2-2]$, define $G^*_i$ as the subgraph of $G[S_i, S_{i+1}]$ obtained by removing all edges that are incompatible with some edge of $M'_{i+1}$ or $M'_{i-1}$.

\begin{claim}\label{claim:almostreg}
    With probability at least $1 - e^{-n^{1/10}}$ we have
    $d_{G_i^*}(v) = (d-\mu)^2 m^2/d \pm n^{7/16}$
    for all even $i \in [m^2-2]$ and $v\in V(G_i^*)$.
\end{claim}

\begin{proof}
    Let $i \in [m^2-2]$ be even, $v \in S_i$, and let $j \in [m]$ be such that $v \in S_i^{(j)}$.
    Let $X$ be the event that $\left| d_{G_i^*}(v) - (d-\mu)^2 m^2/d \right| \geq n^{7/16}$.
    By symmetry and the union bound, it suffices to show that $X$ occurs with probability at most $e^{-n^{1/9}}$.
    Moreover, note that if we can show that
    $\mathbb{P}(X \mid M'_{i-1,j} = M ) \leq e^{-n^{1/9}}$
    for all $M\in \mathcal{M}_{i-1,j}$, then
    \[ \mathbb{P}(X) = \sum_{ M\in \mathcal{M}_{i-1,j} } \mathbb{P}(X \mid M'_{i-1, j} = M) \cdot \mathbb{P}(M'_{i-1,j} = M) \leq e^{-n^{1/9}}.\]
    It therefore suffices to show that $\mathbb{P}(X \mid M'_{i-1,j} = M ) \leq e^{-n^{1/9}}$ for all $M\in \mathcal{M}_{i-1,j}$.
    
    Fix $M\in \mathcal{M}_{i-1,j}$ and suppose that $M'_{i-1, j} = M$. Let $uv$ be the edge of $M$ that is incident with $v$. 
    Recall that $uv$ is compatible with $\left((d-\mu)m \pm 2m^{3/5}\right)m$ distinct edges $vw\in E(G)$ with $w \in S_{i+1}$.
    Moreover, for any $j'\in [m]$ and $w \in N_{G}(v)\cap S_{i+1}^{(j')}$, there are $(d-\mu)m \pm 2m^{3/5}$ matchings in $\mathcal{M}_{i+1,j'}$ whose edge incident to $w$ is compatible with $vw$.
    Thus, 
    \begin{align*}
        \mathbb{E}_{v,i,j,M}&\coloneqq \mathbb{E}[d_{G_i^*}(v) \mid M'_{i-1, j} = M] = ((d-\mu)m \pm 2m^{3/5}) m \cdot ((d-\mu)m \pm 2m^{3/5})\frac{1}{D}\\
        &= \frac{(d-\mu)^2m^3}{dm+2m^{3/5}}  \pm m^{17/10} =\frac{(d-\mu)^2m^2}{d}\pm \frac{n^{7/16}}{2}.
    \end{align*}
    We will now apply McDiarmid's inequality (\cref{thm:mcdiarmid}). Observe that, after conditioning on $M'_{i-1, j} = M$, changing our choice of matching $M'_{i+1, j'}$ for any given $j' \in [m]$ only alters $d_{G_i^*}(v)$ by at most $m$ (whereas changing any other matching does not affect it at all). So we have
    \begin{align*}
        \mathbb{P}\left(X \mid M'_{i-1, j} = M \right)&\leq \mathbb{P}\left(\left| d_{G_i^*}(v) -\mathbb{E}_{v,i,j,M} \right| \geq \frac{n^{7/16}}{2} \quad \mid \quad M'_{i-1, j} = M\right) \\
        &\leq 2 \exp\left( \frac{- n^{7/8}}{4m^3}\right) \leq e^{-n^{1/9}},
    \end{align*}
    as desired.
    \end{proof}

    Therefore, there exists a choice of $\{M_{i'}' \colon i'\in [m^2-1],\, i' \equiv 1 \mod 2\}$ such that $d_{G_i^*}(v)=(d-\mu)^2 m^2/d \pm n^{7/16}$ for all even $i\in [m^2-2]$ and $v\in V(G_i^*)$.
    Since $(d-\mu)^2 m^2/d - n^{7/16} \ge dm^2/25 - n^{7/16}  \geq n^{15/32}$, \cref{lem:almost_matching} (applied with $m^2, (d-\mu)^2m^2/d$, and $n^{7/16}$ playing the roles of $n, d$, and $r$) implies that for each even $i\in [m^2-2]$, the graph $G^*_i$ contains a matching $M'_i$ with at least
    \[m^2 - \frac{2n^{7/16}m^2 }{n^{15/32}} \geq \left(1 - \frac{2}{n^{1/32}} \right)m^2\]
    edges. Recalling \eqref{eq:M} and that $m$ is even, we observe that $\{M_i': i\in [m^2-2]\}$
    induces a compatible linear forest on at least
    \begin{align*}
        (m-n^{1/5})m\cdot \frac{m^2}{2}+\left(1 - \frac{2}{n^{1/32}} \right)m^2\cdot \left( \frac{m^2}2 - 1 \right)\geq\left(1 - \frac{2}{n^{1/32}} \right) m^4\geq n-n^{32/33}
    \end{align*}    
    edges, as desired.
    \end{proof}

\section{Connecting}\label{sec:connecting}

In this section we will prove that given any pair of edges, there is a compatible path joining them that is not too long. This lemma will also be crucial for constructing the absorbing structure in Section~\ref{sec:absorbing}.

\begin{lemma}\label{lem:connecting}
Let $0< 1/n \ll \alpha \ll 1/\ell \ll \rho,\varepsilon$, and let $\delta \geq 1/2 + \eps$ and $\mu \coloneqq (1 - \rho)\delta/2$. Let $G$ be an $n$-vertex graph with $\delta(G)\geq \delta n$, and let $\mathcal{F}$ be a $\mu n$-bounded incompatibility system for $G$. Let $S\subseteq V(G)$ with $|S|\leq \alpha n$. Then, for any disjoint $x_1x_1',x_2x_2' \in E(G)$, there is a compatible path $x_1x_1'P x_2'x_2$ in $G$ of length at most $\ell$ and such that $V(P)\cap S\subseteq \{x_1',x_2'\}$.
\end{lemma}

Here, we sketch the proof idea. 
Our aim is to find a large vertex set~$V^{x_1x'_1}$ such that, for any edge $vu$ with $v \in V^{x_1x'_1}$, there are many short compatible paths $x_1x_1'P vu$.
The lemma will then follow from the fact that $V^{x_1x'_1} \cap V^{x_2x'_2} \ne \emptyset$.
To bound the size of~$V^{x_1x'_1}$ from below, we define an auxiliary digraph~$D_i$, where $uv \in E(D_i)$ if there are many compatible paths $x_1x'_1 v_1 v_2\dots v_{i}$ with $(v_{i-1} v_i) = (u,v)$. 
We show that $D_i$ gets denser and denser as $i$ increases.  
More importantly, we will show that if a vertex $v$ has indegree greater than~$(\mu +\eps)n$, then  $v \in V^{x_1x_1'}$.
A similar idea was used in~\cite{AsymptoticBollobasErdos}.

\begin{proof}[Proof of \cref{lem:connecting}] 
Fix an additional constant $\eta$ such that $1/\ell \ll \eta \ll \rho, \eps$ and denote $t\coloneqq \lceil3/\eta\rceil$.

It will be convenient throughout the proof to view edges as directed, so for clarity we consider the digraph $D$ obtained from $G$ by replacing each undirected edge $uv \in E(G)$ with both directed edges $uv$ and $vu$.
We generalise the notion of compatibility to digraphs in the natural way and, with a slight abuse notation, given two directed edges $uv,vw$, we write $\{uv,vw\}\in \mathcal{F}$ when this holds when considered as undirected edges.

Let $uv\in E(D)$. Denote
\[U^{uv}_1 \coloneqq \{w \in N_D^+(v): \{uv,vw\}\notin \mathcal{F}\}
\text{ and }
V_{1}^{uv} = \emptyset,
\]
and define $D_1^{uv}$ to be the spanning subdigraph of~$D$ such that
\[E(D_1^{uv}) = \{yz \in E(D): y \in U_1^{uv} \mbox{ and } \{vy,yz\}\notin \mathcal{F}\}.\]
For each $i \geq 1$, denote
\begin{align*}
U_{i+1}^{uv} \coloneqq \{w \in V(D): d^-_{D_i^{uv}}(w) \geq \eta n\}
\text{ and }
 V_{i+1}^{uv}  \coloneqq 
\{w \in V(D): d^-_{D_{i}^{uv}}(w) \geq (\mu + \eta)n\},
\end{align*}
and define $D_{i+1}^{uv}$ to be the spanning subdigraph of~$D$ such that
\begin{align*}
       E(D_{i+1}^{uv}) =  E(D_i^{uv})   \cup  \{yz \in E(D):  y \in U_{i+1}^{uv} \mbox{ and } \left| \left\{xy \in E(D_i^{uv}): \{xy,yz\}\notin \mathcal{F}\}\right|\geq \eta^2n \right\}.
\end{align*}
Note that clearly $V_{i}^{uv} \subseteq U_{i}^{uv}$ and $D_{i}^{uv} \subseteq D_{i+1}^{uv}$ for each $i \geq 1$ and this in turn gives $U_{i}^{uv} \subseteq U_{i+1}^{uv}$ for each $i \geq 2$ since $d^-_{D_i^{uv}}(v) \geq \eta n$ implies $d^-_{D_{i+1}^{uv}}(v) \geq \eta n$, and similarly, $V_i^{uv}\subseteq V_{i+1}^{uv}$.

The following claim states that we can construct short compatible paths from $uv$ to any vertex in $V_t^{uv}$.

\begin{claim}\label{claim:path}
    Let $uv, v_0w\in E(D)$ be disjoint with $v_0\in V_t^{uv}$. Then, for any $S'\subseteq V(D)$ with $|S'|\leq \eta^2n/2$, there exists a compatible path $P=uv\dots v_0w$ in $D$ of length at most $t+2$ and such that $V(P)\cap (S\cup S')\subseteq \{u,v,v_0,w\}$.
\end{claim}

\begin{proofclaim}
    Let $i\geq 0$ be maximum such that there exists a compatible path $v_{i}v_{i-1}\dots v_0w$ in $D$ such that $v_i, \dots, v_1\in V(D)\setminus (S\cup S'\cup \{u,v,v_0,w\})$ and $v_iv_{i-1}\in E(D_{i'}^{uv})$ for some $i'\leq t-i$.
    If $v_iv_{i-1}\in E(D_1^{uv})$, then $v_i\in U_1^{uv}$ and so $uvv_{i}v_{i-1}\dots v_0w$ is a compatible path by definition of $U_1^{uv}$ and $D_1^{uv}$, as desired.
    If $v_iv_{i-1}\notin E(D_1^{uv})$, then by definition of $D_{i'}^{uv}$, there is at least $\eta^2n-|S|-|S'|-4-t\geq 1$ vertex $v_{i+1}\in V(D)\setminus (S\cup \{u,v,v_i, \dots, v_0, w\})$ such that $v_{i+1}v_i\in E(D_{i'-1}^{uv})$ and is compatible with $v_iv_{i-1}$, a contradiction to the maximality of $i$.
\end{proofclaim}

Suppose that there exists $v_0 \in \left( V_t^{x_1x_1'}\cap V_t^{x_2x_2'} \right) \setminus (S\cup \{x_1,x_1',x_2,x_2'\})$. Then, applying \cref{claim:path} with $x_1x_1'$ and $\{x_2,x_2'\}$ playing the roles of $uv$ and $S'$, and an arbitrary edge $v_0w\in E(D)$, we obtain a compatible path $P_1=x_1x_1'\dots v_1v_0$ in $D$ of length at most $t+1$ and with $V(P_1)\cap (S\cup S' \cup \{x_2,x_2'\})\subseteq\{x_1,x_1'\}$. 
Now applying \cref{claim:path} again with $x_2x_2',V(P_1)$, and $v_0v_1$ playing the roles of $uv,S'$, and $v_0w$ gives a compatible path $P_2=x_2x_2'\dots v_0v_1$ in $D$ of length at most $t+2$ and such that $V(P_2)\cap (S\cup V(P_1))\subseteq \{x_2,x_2',v_0,v_1\}$. By definition of $D$, the undirected path $x_1x_1'\dots v_1v_0 \dots x_2'x_2$ is a compatible path of $G$, of length at most $2t+2$ and such that $V(P)\cap S\subseteq \{x_1,x_1',x_2,x_2'\}$.
Thus, it suffices to show that $|V_t^{uv}|\geq (1+\varepsilon) n/2$ for all $uv \in E(D)$.

Let $uv\in E(D)$.
For ease of notation, given $i\geq 1$, we write $U_i, V_i, D_i$ in place of $U_i^{uv}, V_i^{uv}, D_i^{uv}$ and denote $W_i \coloneqq U_i \setminus V_i$. 
We start with a few easy observations.  Since $\delta(G)\geq \delta n$ and $\mathcal{F}$ is $\mu n$-bounded, given $xy\in E(D)$, there are at least $(\delta - \mu )n - 1$ vertices $z\in N_D^+(y)$ such that $xy$ and $yz$ are compatible. Together with the definition of~$U_2$, we obtain
\[\frac{(\delta - \mu)^2n^2}{2}\leq |U_1|((\delta- \mu)n -1) \leq e(D_1)\leq |U_2|n + \eta n ^2,\]
which implies that
\begin{equation}\label{eq:sizeofU2}
    |U_2| \geq \frac{(\delta - \mu)^2n}{2} - \eta n \geq 2\eta^{1/2}n.
\end{equation}
Moreover, we claim that for each $i \geq 2$ and $y \in U_{i}$, we have
\begin{equation}\label{eq:outdeginEi}d^+_{D_{i}}(y) \geq d^+_D(y) - (\mu + \rho/4)n \geq (\delta -  \mu - \rho/4)n.\end{equation}
Indeed, suppose not for a contradiction. Then, there are at least $(\mu + \rho/4)n$ edges $yz\in E(D)$ which are incompatible with at least $d^-_{D_{i-1}}(y) - \eta^2 n \geq (1 - \eta)d^-_{D_{i-1}}(y)$ edges $xy\in D_{i-1}$ (where we used that $d^-_{D_{i-1}}(y) \geq \eta n$ by definition of~$U_i$). By averaging, there is some $xy \in E(D_{i-1})$ which is incompatible with at least $(1 - \eta)(\mu + \rho/4)n > \mu n$ edges $yz\in E(D)$, a contradiction.

Observe that if $yz\in E(D)$ with $y\in V_i$, then by definition of~$V_i$ and since $\mathcal{F}$ is $\mu n$-bounded, $yz$ is compatible with at least $\eta n\geq \eta^2 n$ edges $xy\in E(D_{i-1})$ and so $yz\in E(D_i)$. In other words,
\begin{equation}\label{eq:edgesfromYi}D_i[V_i, V(D)] = D[V_i, V(D)]\end{equation}
and so, in particular,
\begin{equation}\label{eq:Xi-Yi}
    e_{D_i}(W_i, V_i) \leq e_D(W_i, V_i) = e_D(V_i, W_i) =  e_{D_i}(V_i, W_i).
\end{equation}

The following claim analyses the growth rates of $|U_{i}|$ and $|V_i|$.

\begin{claim}\label{clm:growth}
    For each $uv\in E(D)$ and $i \geq 2$, one of the following holds. 
\begin{enumerate}[label = \textnormal{(\roman*)}]
    \item $|V_{i+1}| \geq |V_i| + \eta n$;\label{Yeta}
    \item $|U_{i+1}| \geq |U_i| + \eta n$;\label{Ueta}
    \item $|V_i| \geq (1/2 + \eps/2)n$.\label{Ylarge}
\end{enumerate}
\end{claim}

\begin{proofclaim} Let $uv\in E(D)$ and $i\geq 2$. 

We may assume that \cref{Ueta} does not hold, otherwise we are done. Then, by definition of $U_{i+1}$, there are fewer than $\eta n$ vertices $y \in V(D) \setminus U_i$ with $d^-_{D_i}(y) \geq \eta n$ and so $e(D_i[U_i, V(D) \setminus U_i])<\eta n^2 + \eta n^2=2\eta n^2$. Therefore, 
\begin{equation}\label{eq:deginsideUi}
    \mbox{all but at most }\eta^{1/2} n \mbox{ vertices }w \in U_i \mbox{ satisfy } d^+_{D_i}(w, V(D)\setminus U_i) < 2\eta^{1/2} n.
\end{equation}

First suppose that $|W_i| \geq \eta^{1/3}n$.
We may assume that \cref{Yeta} does not hold, for otherwise we are done. Then, by definition of $V_{i+1}$, there are fewer than $\eta n$ vertices $w\in W_i$ such that $d_{D_i}^-(w) \geq (\mu + \eta)n$. Thus, we have, on the one hand, that
\begin{align*}
    e_{D_i}(U_i, W_i) &< \eta n^2 + |W_i|(\mu + \eta)n  \leq |W_i|(\mu + 2\eta^{2/3})n\leq |W_i|(\delta - \mu - \rho/3)n .
\end{align*}
On the other hand, since $W_i = U_i \setminus V_i$, we have
\begin{align*}
        e_{D_i}(U_i, W_i) &= e_{D_i}(W_i, W_i) + e_{D_i}(V_i, W_i) \stackrel{\eqref{eq:Xi-Yi}}{\geq} e_{D_i}(W_i, W_i) + e_{D_i}(W_i, V_i)\\ 
        &= e_{D_i}(W_i, U_i)     \stackrel{\eqref{eq:outdeginEi},\eqref{eq:deginsideUi}}{\geq} (|W_i| - \eta^{1/2}n)(\delta - \mu - \rho/4 - 2\eta^{1/2})n \\
        &\geq (1- \eta^{1/6}) |W_i|(\delta - \mu - \rho/4 - 2\eta^{1/2})n
        \geq |W_i|(\delta - \mu - \rho/3)n ,
\end{align*}
a contradiction.

We may therefore assume that $|W_i| \leq \eta^{1/3}n$. 
We may also assume that \cref{Ylarge} does not hold, for otherwise we are done.
Together with \eqref{eq:edgesfromYi}, this implies that for each $w\in V_i\subseteq U_i$,
\begin{align*}
d^+_{D_i}(w, V(D) \setminus U_i) & = d^+_D(w, V(D)\setminus (V_i \cup W_i)) \geq \delta n - |V_i| - |W_i| \\
& \geq \delta n-(1/2 + \eps/2)n-\eta^{1/3}n\geq 2\eta^{1/2}n.
\end{align*}
But
\[|V_i| = |U_i \setminus W_i| \stackrel{\eqref{eq:sizeofU2}}{\geq} 2\eta^{1/2}n - \eta^{1/3}n > \eta^{1/2}n,\]
a contradiction to \eqref{eq:deginsideUi}.
This completes the proof of the claim. 
\end{proofclaim}

\Cref{clm:growth} implies that, for each $uv\in E(D)$, we have $|V_t^{uv}| \geq (1/2 + \eps/2)n$, as desired. 
Indeed, otherwise we have
$2n\geq |U^{uv}_t| + |V_t^{uv}| \geq (t-2) \eta n > \frac{2}{\eta} \cdot \eta n = 2n$,
a contradiction.
This completes the proof of \cref{lem:connecting}.\end{proof}

\section{Absorbing lemma}\label{sec:absorbing}

We will now prove an absorbing lemma. The idea is that we can set aside a small set of vertices such that given a compatible linear forest without too many components, we can glue together the paths of the linear forest into a compatible Hamilton cycle using the reserved vertices.

\begin{lemma} \label{lem:absorbing}
Let $1/n \ll \beta \ll \alpha \ll \rho, \eps$.
Let $\delta \ge 1/2+ \eps$ and $\mu \coloneqq (1- \rho) \delta/2$. 
Suppose that $G$ is an $n$-vertex graph with $\delta(G) \ge \delta n$ and $\mathcal{F}$ is a $\mu n$-bounded incompatibility system for~$G$. 
Then there exists~$A\subseteq V(G)$ with $|A|\leq \alpha n$ such that, given any compatible linear forest~$H$ in $G \setminus A$ with at most $\beta n$ components, there exists an $\mathcal{F}$-compatible cycle with vertex set $A \cup V(H)$.
\end{lemma}

Note that the statement of \cref{lem:absorbing} differs somewhat from the typical absorbing lemma for Hamilton cycles, as seen in e.g.~\cite{rodl2008approximate}: Usually, one reserves an absorber $A$, as well as a separate small set of vertices $R$ called a reservoir. Then, $R$ and the connecting lemma (\cref{lem:connecting}) are first used to glue all the paths in the linear forest $H$ into a long cycle (the connecting step). Finally, $A$ is used to transform this cycle into a Hamilton cycle (the absorbing step). For cosmetic reasons, we merge these two steps into a single statement where the absorber $A$ directly turns any linear forest into a Hamilton cycle.

\bigskip

In this section, we use the following conventions. We assume that all paths are an ordered sequence of vertices, so that $v_1 \dots v_{\ell}$ and $v_{\ell} \dots v_1$ are distinct (when $\ell \ge 2$). 
If $P = v_1 \dots v_\ell$ we write $P^{-1}$ for $v_{\ell} \dots v_1$.

Let $P = v_1 \dots v_\ell$ be a path. 
We say that $v_1v_2$ and $v_{\ell} v_{\ell -1}$ are the \emph{end-edges} of~$P$. 
Note that implicitly there is an ordering on the vertices of $v_1v_2$ and $v_\ell v_{\ell -1 }$ and we will always write these edges in this way with the endpoints of $P$ first.

\bigskip

The proof of the absorbing lemma follows a fairly standard approach. Given a path~$Q$ in the linear forest, an absorbing structure will be itself a compatible path $P$ that can be altered to form a new compatible path with the same end-edges as $P$ that covers all of $V(P) \cup V(Q)$. 
The vertex set $A$ that we set aside will be the vertices of a compatible cycle made by connecting together absorbing paths. Suppose that for each compatible linear forest on $G \setminus A$ we can pair each component $Q$ of the linear forest with a suitable absorbing path $P$ in this cycle. Then by definition we can alter $P$ to a path with the same end-edges that also covers all the vertices of the corresponding path $Q$. Doing this absorption for each component of the linear forest in turn gives a compatible Hamilton cycle on $G$.

We first prove that given any pair of edges (the end-edges for a non-trivial path in the linear forest) or any vertex (for trivial components of the linear forest) there exists an absorbing path. By a supersaturation result we further show that there are many such absorbing paths for every pair of edges and every vertex. Then taking absorbing paths at random and gluing them into a not-too-large cycle gives the result.

\subsection{Absorbing structures}

Let $P$ be a path with end-edges~$v_1v'_1$ and~$v_2v'_2$.
For edges $x_1x'_1, x_2x_2'$, we say that $P$ is an \emph{absorbing path for}~$(x_1x'_1, x_2x_2')$ if $V(P) \cap \{x_1,x'_1, x_2,x_2'\} = \emptyset$, $P$ is compatible and there exist compatible paths $P_1$, $P_2$ such that $V(P_1) \cup V(P_2) = V(P) \cup \{x_1,x'_1,x_2,x_2'\}$, $V(P_1) \cap V(P_2) = \{x_1,x_1'\}\cap \{x_2,x_2'\}$ and each $P_i$ has end-edges $v_iv'_i$ and $x_ix'_i$.

We use these absorbing paths to absorb a path~$Q$ with end-edges $x_1x'_1$ and $x_2x_2'$. 
The condition $V(P_1) \cap V(P_2) = \{x_1,x_1'\}\cap \{x_2,x_2'\}$ is to deal with the case when $Q$ has length~$2$ or fewer. 

\begin{fact} \label{fact:absorb}
Let $G$ be a graph and  $\mathcal{F}$ be an incompatibility system for~$G$. 
Let $Q$ be a compatible path with end-edges $x_1x'_1$ and $x_2x_2'$. 
Let $P$ be an absorbing path for $(x'_1x_1,x'_2x_2)$ with $V(P) \cap V(Q) = \emptyset$.
Then there exists a compatible path~$P^*$ on $V(Q) \cup V(P)$ with the same end-edges as~$P$.
\end{fact}

\begin{proof}
Let $v_1v'_1$ and~$v_2v'_2$ be end-edges of~$P$. 
By the definition of~$P$, there exists compatible paths $P_1$, $P_2$ such that $V(P_1) \cup V(P_2) = V(P) \cup \{x_1,x'_1,x_2,x_2'\}$, $V(P_1) \cap V(P_2) = \{x_1,x'_1\} \cap \{x_2,x_2'\}$ and each $P_i$ has end-edges $v_iv'_i$ and $x'_ix_i$.
We are done by setting $P^* = P_1QP_2^{-1}$.
\end{proof}

Similarly for a vertex~$x$, we say that $P$ is an \emph{absorbing path for}~$x$ if $x \notin V(P)$, $P$ is compatible and there exists a compatible path~$Q$ such that $V(Q) = V(P) \cup \{x\}$ and $Q$ has the same end-edges as~$P$.
The proof of the following lemma is inspired by~\cite[Section~4]{LoPatel}.

\begin{lemma} \label{lem:absorbinggadget}
Let $1/n \ll 1/\ell \ll \rho, \eps$.
Let $\delta \ge 1/2+ \eps$ and $\mu \coloneqq (1- \rho) \delta/2$. 
Suppose that $G$ is an $n$-vertex graph with $\delta(G) \ge \delta n$ and $\mathcal{F}$ is a $\mu n$-bounded incompatibility system for~$G$. 
Then for any $x_1x'_1, x_2x_2'\in E(G)$, there exists an $\mathcal{F}$-compatible absorbing path for~$(x_1x'_1, x_2x_2')$ of length~$3\ell$.
For any $x\in V(G)$, there exists an $\mathcal{F}$-compatible absorbing path for~$x$ of length~$3\ell$.
\end{lemma}

\begin{proof}
Let $x_1x'_1, x_2 x_2'\in E(G)$. 
Since $\delta(G)-\mu n>n/4$, there exist distinct $y_1,y_2\in V(G) \setminus \{x_1,x'_1, x_2,x_2'\}$ such that $x_ix'_iy_i$ is a compatible path in $G$ for each $i\in [2]$.
Since $\delta(G)\geq(1/2 + \eps)n$, there exist distinct $z_1,z_2,w_1,w_2 \in V(G) \setminus \{x_1,x'_1, x_2,x_2',y_1,y_2\}$ such that $y_1z_1w_1w_2z_2y_2$ is a (not necessarily compatible) path in~$G$.
Since $\delta(G)-2\mu n\ge \rho n /2$, there exist distinct $y'_1,y'_2,z'_1,z'_2,w'_1,w'_2\in V(G) \setminus \{x_1,x'_1, x_2,x_2',y_1,y_2,z_1,z_2,w_1,w_2\}$ such that, for all $i \in [2]$,
\begin{align*}
	\{ y'_i y_i, y_i z_i\}, \{ y'_i y_i, y_i x_i'\},
	\{ z'_i z_i, z_i y_i\}, \{ z'_i z_i, z_i w_i\},
	\{ w'_i w_i, w_i z_i\}, \{ w'_i w_i, w_1w_2\} \notin \mathcal{F}.
\end{align*}
Let $S \coloneqq \{x_i,x_i',y_i,y'_i, z_i, z_i', w_i, w'_i : i \in [2]\}$.

We construct our absorbing path as follows (see also Figure~\ref{fig:absorbing}). Apply~\cref{lem:connecting} with $z_2z_2'$ and $z_1z_1'$ playing the roles of $x_1x_1'$ and $x_2x_2'$ to obtain a compatible path $z_2z'_2 Q_1 z'_1 z_1$ of $G$, of length at most $\ell$ and such that $V(Q_1)\cap S=\{z_2',z_1'\}$. Apply \cref{lem:connecting} again with $w_2w_2', y_2y_2'$, and $S\cup V(Q_1)$ playing the roles of $x_1x_1',x_2x_2'$, and $S$ to obtain a compatible path $w_2 w'_2 Q_2 y_2' y_2$ of $G$, of length at most $\ell$ and such that $V(Q_2)\cap (S\cup V(Q_1))=\{w_2',y_2'\}$. 
Since $\delta(G)-\mu n> n/4$, we can greedily extend $w_1w'_1$ into a compatible path~$w_1w_1'Q_3u'u$ of length $3 \ell - (8 + e(Q_1) +e(Q_2))$, where $u u'$ is an arbitrary edge and such that $V(Q_3)\cap (S\cup V(Q_1)\cup V(Q_2))=\{w_1'\}$.

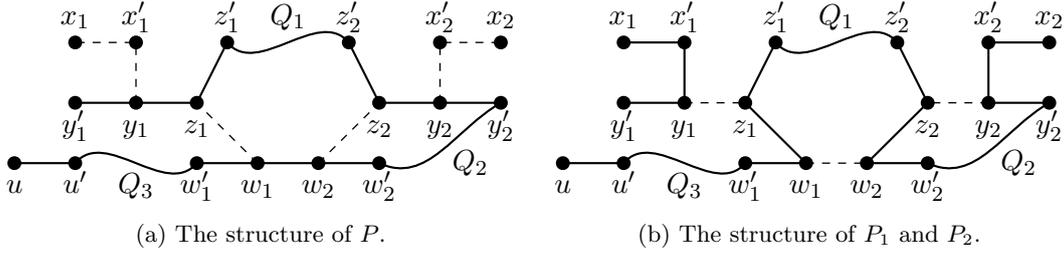
\begin{figure}[ht]
\centering
\subfloat[\label{fig:P}The structure of $P$.]{
\begin{tikzpicture}[scale =.8, line width=2pt]
	
   \foreach \x/\y/\z in {1/0/w, 2/0/{w'}, 2/1/z , 3/1/y , 4/1/{y'}}
			{
			\filldraw[fill=black] ({\x-0.5},\y) circle (2pt);
			\node[anchor=base]  at ({\x-0.5},{\y-0.5}) {$\z_2$};
			\filldraw[fill=black] ({-\x+0.5},\y) circle (2pt);
			\node[anchor=base]  at ({-\x+0.5},{\y-0.5}) {$\z_1$};
			}
   \foreach \x/\y/\z in {1.5/2/{z'}, 3/2/{x'},4/2/x}
			{
			\filldraw[fill=black] ({\x-0.5},\y) circle (2pt);
			\node[anchor=base]  at ({\x-0.5},{\y+0.3}) {$\z_2$};
			\filldraw[fill=black] ({-\x+0.5},\y) circle (2pt);
			\node[anchor=base]  at ({-\x+0.5},{\y+0.3}) {$\z_1$};
			}

		\begin{scope}[line width=0.5pt, dashed]
		\draw (-0.5,0)--(-1.5,1);	
		\draw (-2.5,1)--(-2.5,2)--(-3.5,2);	
		\draw (0.5,0)--(1.5,1);	
		\draw (2.5,1)--(2.5,2)--(3.5,2);	
		\end{scope}
		
		\begin{scope}[thick]
		\draw (-1.5,0)--(1.5,0);	
		\draw (-3.5,1)--(-2.5,1)--(-1.5,1)--(-1,2);	
		\draw (3.5,1)--(2.5,1)--(1.5,1)--(1,2);	
		\end{scope}
		
		\draw[thick] (-1,2) to[out = -50, in = -230] (1,2);
		\node[anchor=base]  at (0,2.3) {$Q_1$};
		
		\draw[thick] (1.5,0) to[out = -25, in = -150] (3.5,1);
		\node  at (3,0) {$Q_2$};

		\filldraw[fill=black] (-3.5,0) circle (2pt);
		\filldraw[fill=black] (-4.5,0) circle (2pt);
		\node[anchor=base]  at (-3.5,-0.5) {$u'$};
		\node[anchor=base]  at (-4.5,-0.5) {$u$};
		\draw[thick] (-3.5,0)--(-4.5,0);
		\draw[thick] (-3.5,0) to[out = 50, in = 230] (-1.5,0);
		\node[anchor=base]  at (-2.5,-0.5) {$Q_3$};

\end{tikzpicture}
}
\subfloat[\label{fig:P1P2}The structure of $P_1$ and $P_2$.]{
\begin{tikzpicture}[scale =.8, line width=2pt]
		
   \foreach \x/\y/\z in {1/0/w, 2/0/{w'}, 2/1/z , 3/1/y , 4/1/{y'}}
			{
			\filldraw[fill=black] ({\x-0.5},\y) circle (2pt);
			\node[anchor=base]  at ({\x-0.5},{\y-0.5}) {$\z_2$};
			\filldraw[fill=black] ({-\x+0.5},\y) circle (2pt);
			\node[anchor=base]  at ({-\x+0.5},{\y-0.5}) {$\z_1$};
			}
   \foreach \x/\y/\z in {1.5/2/{z'}, 3/2/{x'},4/2/x}
			{
			\filldraw[fill=black] ({\x-0.5},\y) circle (2pt);
			\node[anchor=base]  at ({\x-0.5},{\y+0.3}) {$\z_2$};
			\filldraw[fill=black] ({-\x+0.5},\y) circle (2pt);
			\node[anchor=base]  at ({-\x+0.5},{\y+0.3}) {$\z_1$};
			}
		\begin{scope}[thick]
		\draw (-1.5,0)--(-0.5,0)--(-1.5,1)--(-1,2);	
		\draw (-3.5,1)--(-2.5,1)--(-2.5,2)--(-3.5,2);	
		\draw (1.5,0)--(0.5,0)--(1.5,1)--(1,2);	
		\draw (3.5,1)--(2.5,1)--(2.5,2)--(3.5,2);	
		\end{scope}
		
		\begin{scope}[line width=0.5pt, dashed]
		\draw (-.5,0)--(.5,0);	
		\draw (-2.5,1)--(-1.5,1);	
		\draw (2.5,1)--(1.5,1);	
		\end{scope}
		
		\draw[thick] (-1,2) to[out = -50, in = -230] (1,2);
		\node[anchor=base]  at (0,2.3) {$Q_1$};
		
		\draw[thick] (1.5,0) to[out = -25, in = -150] (3.5,1);
		\node  at (3,0) {$Q_2$};

		\filldraw[fill=black] (-3.5,0) circle (2pt);
		\filldraw[fill=black] (-4.5,0) circle (2pt);
		\node[anchor=base]  at (-3.5,-0.5) {$u'$};
		\node[anchor=base]  at (-4.5,-0.5) {$u$};
		\draw[thick] (-3.5,0)--(-4.5,0);
		\draw[thick] (-3.5,0) to[out = 50, in = 230] (-1.5,0);
		\node[anchor=base]  at (-2.5,-0.5) {$Q_3$};
\end{tikzpicture}
}
\caption{An absorbing path for $(x_1x_1',x_2x_2')$.}
\label{fig:absorbing}
\end{figure}

Let 
\begin{align*}
P \coloneqq y'_1 y_1 z_1 z_1' Q_1 z'_2 z_2 y_2 y_2' Q_2 w'_2 w_2 w_1w_1' Q_3 u' u,
\end{align*} 
as in Figure~\ref{fig:absorbing}(a).
Note that $P$ is a compatible path of length~$3\ell$ with end-edges~$y'_1y_1$ and~$u u'$.
Then,
\begin{align*}
P_1 \coloneqq y'_1 y_1 x'_1 x_1
\qquad \text{ and } \qquad
P_2 \coloneqq uu' Q_3^{-1} w_1'w_1z_1 z'_1 Q_1^{-1} z_2' z_2 w_2 w_2' Q_2^{-1} y'_2y_2 x'_2x_2
\end{align*}
(see Figure~\ref{fig:absorbing}(b)) witness that $P$ is an absorbing path for~$(x_1x'_1, x_2x_2')$.

Let $x \in V(G)$. An absorbing path for~$x$ can be constructed similarly. 
Let $y_1, y_2\in N_G(x)$ be distinct such that $y_1xy_2$ is compatible.
Choose distinct $z_1,\allowbreak z_2,\allowbreak w_1,\allowbreak w_2,\allowbreak y_1', \allowbreak y_2',\allowbreak z_1',\allowbreak z_2',\allowbreak w_1',\allowbreak w_2'\in V(G)$ satisfying the same properties as above, with $x$ playing the role of $x_1'$ and $x_2'$.
Using \cref{lem:connecting}, we can construct vertex-disjoint paths $Q_1,Q_2,Q_3$ with the same properties as above. Then,
\begin{align*}
Q \coloneqq y'_1 y_1 x y_2 y'_2 Q_2 w_2' w_2 z_2 z_2' Q_1 z'_1z_1 w_1w_1' Q_3u' u
\end{align*}
witnesses that
\begin{align*}
P \coloneqq y'_1 y_1 z_1 z_1' Q_1 z'_2 z_2 y_2 y_2' Q_2 w'_2 w_2 w_1w_1' Q_3 u' u
\end{align*} 
is an absorbing path for $x$.
\end{proof}

\subsection{Supersaturation}

By examining the proofs of Lemmas~\ref{lem:absorbinggadget} and~\ref{lem:connecting}, one could deduce that there are many absorbing paths, see Corollary~\ref{cor:absorbinggadget} for the statement.
However, counting the number of ways we have for constructing (absorbing) paths in those lemmas is cumbersome, so we use the following result instead, whose proof is based on a result of Mubayi and Zhao~\cite[Lemma 2.1]{MubayiZhao}.

\begin{lemma} \label{lem:randomsubset}
Let $0 <1/n \ll 1/m \ll \eps \ll  \mu, 1/s $, and let $\delta\geq \mu/2$.
Suppose that $G$ is an $n$-vertex graph with $\delta(G) \ge \delta n$ and $\mathcal{F}$ is a $\mu n$-bounded incompatible system for~$G$. 
Let $S \subseteq V(G)$ be such that $|S| =s$. 
Then, there are at least $\binom{n-s}{m-s}/2$ sets $M \subseteq V(G) \setminus S$ of size $m-s$ such that $\delta (G[M \cup S]) \ge (\delta - \eps) m$ and $\mathcal{F}[M \cup S]$ is a $(\mu+\eps)m$-bounded incompatibility system for~$G[M \cup S]$.
\end{lemma}

\begin{proof}
We say that $M$ is \emph{bad for} a vertex $x$ if $x \in M \cup S$ and $d_G(x,M \cup S) \le (\delta - \eps)m$.
Let $x\in V(G)$. Note that $d_G(x,V(G)\setminus S) \ge \delta n - s \ge (\delta- \eps/4) (n-s)$. 
By Corollary~\ref{cor:chernoff2} (applied with $n-s, N_G(x), \delta-\varepsilon/2, \varepsilon/4, m - |\{x\} \cup  S|$ playing the roles of $n, Y, \alpha, \lambda,m$), we have 
\begin{align*}
	& \left| \left\{ M \in \binom{V(G)\setminus S}{m-s} \colon \text{$M$ is bad for $x$} \right\} \right| \\
    & \qquad  \le 
    \left| \left\{ M \in \binom{V(G)\setminus (\{x\} \cup S)}{m-|\{x\} \cup  S|} \colon d_G(x,M) \le (\delta - \eps)m \right\} \right| 
	\\ 
    & \qquad  \le 
    \left| \left\{ M \in \binom{V(G)\setminus (\{x\} \cup S)}{m-|\{x\} \cup  S|} \colon d_G(x,M) \le (\delta - \eps/2)(m-|\{x\} \cup  S|) \right\} \right| 
	\\ 
    &\qquad \le 
	\begin{cases}
	2\binom{n-s}{m-s} e^{-(\eps/4)^2 (m-s) /3} 
			& \text{if $x \in S$}\\
	2\binom{n-s-1}{m-s-1} e^{-(\eps/4)^2 (m-s-1) /3} 
			& \text{if $x \in V(G) \setminus S$}
	\end{cases}\\
    &\qquad \leq \dbinom{n-s}{m-s}\frac{1}{12} \cdot \begin{cases}
	\frac{1}{s} 
			& \text{if $x \in S$}\\
	\frac{1}{n}
			& \text{if $x \in V(G) \setminus S$}.
	\end{cases}
\end{align*}

Consider an (ordered) edge~$uv$ in~$G$.
Let $Y(uv)$ be the set of vertices $w \in V(G)\setminus (S \cup \{u,v\})$ such that $\{uv,uw\} \notin \mathcal{F}$, 
so $|Y(uv)| \ge (1- \mu) n -s-2 \ge (1-\mu -\eps/4) (n-s)$.  
We say that $M$ is \emph{bad for}~$uv$ if $u,v \in M \cup S$ and $|(M \cup S) \cap Y(uv)| \le (1- \mu-\eps)m$.
If $M$ is not bad for~$uv$, then there are at most $(\mu + \eps )m$ vertices $w \in M \cup S$ such that $\{uv,uw\} \in \mathcal{F}$. 
Similarly, by Corollary~\ref{cor:chernoff2} (applied with $n-s, Y(uv), 1-\mu-\varepsilon/2, \varepsilon/4,m-|\{u,v\} \cup S|$ playing the roles of $n, X, \alpha, \lambda,m$), we have
\begin{align*}
	&\left| \left\{M \in \binom{V(G)\setminus S}{m-s} : \text{$M$ is bad for $uv$} \right\} \right|
	\\
    & \qquad\le 
    \left| \left\{M \in \binom{V(G)\setminus (\{u,v\} \cup S)}{m-|\{u,v\} \cup S|} : |M  \cap Y(uv)| \le (1- \mu-\eps/2)(m-|\{u,v\} \cup S|) \right\} \right|
	\\
    &\qquad\le
		\begin{cases}
	2\binom{n-s}{m-s} e^{-(\eps/4)^2 (m-s) /3} 
			& \text{if $|\{u,v\} \cap S | =2$}\\
	2\binom{n-s-1}{m-s-1} e^{-(\eps/4)^2 (m-s-1) /3} 
			& \text{if $|\{u,v\} \cap S | =1$}\\
	2\binom{n-s-2}{m-s-2} e^{-(\eps/4)^2 (m-s-2) /3}
			& \text{if $|\{u,v\} \cap S | = 0$}
	\end{cases}
    \\
    &\qquad\le
		\binom{n-s}{m-s} \frac{1}{12} \cdot \begin{cases}
	\frac{1}{s^2} 
			& \text{if $|\{u,v\} \cap S | =2$}\\
	\frac{1}{sn} 
			& \text{if $|\{u,v\} \cap S | =1$}\\
	\frac{1}{n^2}
			& \text{if $|\{u,v\} \cap S | = 0$}.
	\end{cases}
\end{align*}
Therefore, the number of $M\in \binom{V(G)\setminus S}{m-s}$ which are not bad for any vertex or edge is at least
\begin{align*}
	\left( 1 - s \cdot\frac{1}{12s} - n \cdot\frac{1}{12n} - s^2 \cdot\frac{1}{12s^2} - 2sn \cdot\frac{1}{12sn} - n^2 \cdot \frac{1}{12n^2}\right) \binom{n-s}{m-s} \geq \frac{1}{2}\binom{n-s}{m-s} .
\end{align*}
For each such $M$, we have $\delta(G[M\cup S]) \ge (\delta - \eps)m$ and $\mathcal{F}[M \cup S]$ is $(\mu + \eps) m$-bounded. 
\end{proof}

Using Lemma~\ref{lem:randomsubset}, we now prove a supersaturation version of Lemma~\ref{lem:absorbinggadget}.
For a vertex $x \in V$, let $\mathcal{L}_{\ell}(x)$ denote the set of absorbing paths for~$x$ of length~$\ell$.
For edges $x_1x'_1, x_2x_2'$, let $\mathcal{L}_{\ell}(x_1x'_1, x_2x_2')$ denote the set of absorbing paths for~$(x_1x'_1, x_2x_2')$ of length~$\ell$.

\begin{corollary} \label{cor:absorbinggadget}
Let $1/n \ll \beta \ll 1/\ell \ll \rho, \eps$ and suppose $3$ divides~$\ell$.
Let $\delta \ge 1/2+ \eps$ and $\mu \coloneqq  (1- \rho) \delta/2$. 
Suppose that $G$ is an $n$-vertex graph with $\delta(G) \ge \delta n$ and $\mathcal{F}$ is a $\mu n$-bounded incompatibility system for~$G$. 
Then, for any~$x\in V(G)$ and~$x_1x'_1, x_2x_2'\in E(G)$, we have $|\mathcal{L}_\ell(x)|, |\mathcal{L}_\ell(x_1x'_1, x_2x_2')| \ge 8\ell \sqrt\beta n^{\ell}$.
\end{corollary}

\begin{proof}
Fix an additional constant $m$ such that $ \beta \ll 1/m \ll 1/\ell$. We may assume without loss of generality that $\rho \leq 4\varepsilon$. 
Let $\delta' \coloneqq \delta - \rho/8$ and $\mu' \coloneqq ( 1- \rho/4)\delta'/2$.
Note that $\delta' \ge 1/2+ \eps/2$ and 
\begin{align*}
	\mu+ \rho/8 
	\le \frac{(1-\rho)\delta'+ \rho/8}{2} + \frac{\rho}8 
	= (1- \rho) \frac{\delta'}2 + \frac{3 \rho}{16} 
	\le (1- \rho) \frac{\delta'}2 + \frac{3 \rho \delta'}{8}
	= \mu'.
\end{align*}
Fix $x \in V(G)$. 
Let $\mathcal{M}$ be the set of $M \in \binom{V(G) \setminus \{x\}}{m-1}$ such that $\delta(G[M \cup \{x\}]) \ge \delta' m$ and $\mathcal{F}[M \cup \{x\}]$ is $\mu' m$-bounded.
By Lemma~\ref{lem:randomsubset}, applied with $x$ and $\rho/8$ playing the roles of $S$ and $\varepsilon$, we have $|\mathcal{M}| \ge \binom{n-1}{m-1}/2$. 
For each $M \in \mathcal{M}$, Lemma~\ref{lem:absorbinggadget} (applied with $G[M\cup \{x\}], \delta', \rho/4, \varepsilon/2$, and $\mu'$ playing the roles of $G, \delta, \rho, \varepsilon$, and $\mu$) implies that there is an absorbing path for~$x$ of length~$\ell$ in $G[M]$ (recall that an absorbing path for~$x$ does not contain $x$).
For any path $P$ of length $\ell$, there are at most $\binom{n-2-\ell}{m-2-\ell}$ sets $M \in \mathcal{M}$ such that $P\subseteq G[M]$.
Hence,
\begin{align*}
	|\mathcal{L}_{\ell}(x)| \ge \frac{|\mathcal{M}| }{\binom{n-2-\ell}{m-2-\ell}} \ge \frac{ \frac12 \binom{n-1}{m-1}}{\binom{n-2-\ell}{m-2-\ell}} = \frac12 \frac{n-1}{m-1}\frac{ \binom{n-2}{\ell} }{ \binom{m-2}{\ell} } \ge 8\ell \sqrt \beta n^{\ell}.
\end{align*}
A corresponding argument implies that for any~$x_1x'_1, x_2x_2'\in E(G)$, we have 
\begin{align*}
	|\mathcal{L}_{\ell}(x_1x'_1, x_2x_2')| \ge \frac{ \frac12 \binom{n-4}{m-4}}{\binom{n-5-\ell}{m-5-\ell}} = \frac12 \frac{n-4}{m-4}\frac{ \binom{n-5}{\ell} }{ \binom{m-5}{\ell} } \ge 8\ell \sqrt \beta n^{\ell},
\end{align*}
as desired.
\end{proof}

\subsection{Proof of Lemma~\ref{lem:absorbing}}

Using a simple random argument, we obtain the following result. 

\begin{lemma} \label{lem:absorbingset}
Let $1/n \ll \beta \ll \alpha \ll 1/\ell \ll \rho, \eps$ and suppose $3$ divides~$\ell$.
Let $\delta \ge 1/2+ \eps$ and $\mu \coloneqq (1- \rho) \delta/2$. 
Suppose that $G$ is an $n$-vertex graph with $\delta(G) \ge \delta n$ and $\mathcal{F}$ is a $\mu n$-bounded incompatibility system for~$G$. 
Then there exists a family $\mathcal{Q}$ of vertex-disjoint $\mathcal{F}$-compatible paths each of length~$\ell$ such that 
\begin{align*}
	| \mathcal{Q}| &\le \alpha n, &
	| \mathcal{Q} \cap \mathcal{L}_{\ell}(x)| &\ge \beta n,&
	| \mathcal{Q} \cap \mathcal{L}_{\ell}(x_1x_1',x_2x_2')| &\ge \beta n
\end{align*}
for all $ x \in V(G)$ and edges $x_1x_1',x_2x_2' \in E(G)$.
\end{lemma}

\begin{proof}
Let $p  \coloneqq \sqrt{\beta} n^{-(\ell-1)}/\ell$.
Let $\mathcal{Q}'$ be a random subset of $V(G)^{\ell}$, where each $S \in V(G)^{\ell}$ is chosen with probability~$p$ independently. 
For each $S=(v_1, \dots, v_\ell) \in V(G)^{\ell}$, we also view $S$ as the path $v_1\dots v_{\ell}$ (which may not exist in~$G$). 
By Corollary~\ref{cor:absorbinggadget}, we have
\begin{align*}
    \mathbb{E} [|\mathcal{Q}' \cap \mathcal{L}_{\ell}(x)|],\, \mathbb{E} [|\mathcal{Q}' \cap \mathcal{L}_{\ell}(x_1x_1', x_2x_2')|]  \ge p \cdot 8\ell  \sqrt \beta n^{\ell} = 8 \beta n 
\end{align*}
for any $x\in V(G)$ and~$x_1x'_1, x_2x_2'\in E(G)$.
By Chernoff's inequality (Lemma~\ref{lm:chernoff}), we have with probability at least $3/4$ that
\begin{align}
	| \mathcal{Q}'| &\le \sqrt{\beta} n \le \alpha n , &
	| \mathcal{Q}' \cap \mathcal{L}_{\ell}(x)| &\ge 3\beta n ,&
	| \mathcal{Q}' \cap \mathcal{L}_{\ell}(x_1x_1',x_2x_2')| &\ge 3\beta n
    \label{eqn:absorbingset}
\end{align}
for any~$x\in V(G)$ and~$x_1x'_1, x_2x_2'\in E(G)$.

We say that two members of $V(G)^{\ell}$ form an \textit{intersecting pair} if they share at least one vertex (not necessarily in the same position). 
Let $Z$ be the number of intersecting pairs in~$\mathcal{Q}'$, so 
$
\mathbb{E} [Z]  \le p^2  \cdot \ell n^{\ell}  \cdot \ell n^{\ell-1}
    \le \beta n.
$
By Markov's inequality (\cref{lm:Markov}), 
$   \mathbb{P}( Z \ge 2\beta n) \le 
    \mathbb{P}( Z \ge 2 \mathbb{E} [Z]) \le  1/2$.
Thus, there exists $\mathcal{Q}' \subseteq V(G)^{\ell}$ such that $Z \le 2\beta n $ and \eqref{eqn:absorbingset} holds. 
We obtain $\mathcal{Q}$ from~$\mathcal{Q}'$ by deleting one member from each interesting pair and those that are not $\mathcal{F}$-compatible paths in~$G$. 
\end{proof}

We now prove Lemma~\ref{lem:absorbing}.

\begin{proof}[Proof of Lemma~\ref{lem:absorbing}]
Let $1/n \ll \beta \ll \alpha' \ll \alpha \ll 1/\ell \ll \rho , \eps$.
Apply Lemma~\ref{lem:absorbingset} with $\alpha'$ playing the role of $\alpha$ to obtain a family $\mathcal{Q}$ of vertex-disjoint $\mathcal{F}$-compatible paths each of length~$\ell$ such that 
\begin{align*}
	| \mathcal{Q}| &\le \alpha' n, &
	| \mathcal{Q} \cap \mathcal{L}_{\ell}(x)| &\ge \beta n,&
	| \mathcal{Q} \cap \mathcal{L}_{\ell}(x_1x_1',x_2x_2')| &\ge \beta n
\end{align*}
for all $ x \in V(G)$ and edges $x_1x_1',x_2x_2' \in E(G)$.

Denote $\mathcal{Q} = \{ Q_1, Q_2, \dots, Q_q\}$ and for each $i\in [q]$ let $x_ix_i'$ and $y_iy_i'$ be the end-edges of $Q_i$. 
For each $i \in [q]$ in turn, we apply Lemma~\ref{lem:connecting}, with $y_i'y_i, x_{i+1}'x_{i+1},$ and $V(\mathcal{Q})\cup V(Q_1')\cup \dots V(Q_{i-1}')$ playing the roles of $x_1x_1',x_2x_2'$, and $S$, to obtain compatible path $y_i'y_iQ_i'x_{i+1}x_{i+1}'$ of length at most~$\ell$ such that $Q_i Q'_i Q_{i+1}$ is a compatible path (with $Q_{q+1}\coloneqq Q_1$). 
Furthermore, $Q_1, Q'_1, \dots, Q_q,Q'_q$ are vertex-disjoint. 
Let $C$ be the compatible cycle $Q_1Q'_1Q_2Q'_2 \dots Q_qQ_q'$ and $A = V(C)$. 
Note that $|A| \le 2 \ell |\mathcal{Q}| \le 2 \ell \alpha' n \le \alpha n$. 

We now show that $A$ has the desired property. 
Consider any compatible linear forest~$H$ in $G \setminus A$ with at most $\beta n$ components.
Let the components of $H$ be vertices $z_1, \dots, z_{s}$ and paths $P_{s+1}, \dots, P_{s+t}$ of length at least~$1$. 
For each $j \in [t]$, let $z_{s+j} z'_{s+j}$ and $w_{s+j} w'_{s+j}$ be the end-edges of~$P_{s+j}$. 
By \cref{lem:absorbingset}, there exist distinct $Q_{i_1}, Q_{i_2}, \dots, Q_{i_{s+t}}\in \mathcal{Q}$ such that $Q_{i_j}$ is an absorbing path for $z_j$ if $j \in [s]$ and for $(z'_{s+j}z_{s+j}, w'_{s+j}w_{s+j})$
if $j \in [s+1, s+t]$. 
For $j \in [s]$, there exists a compatible path~$Q^*_{i_j}$ on $V(Q_{i_j}) \cup z_j$ with the same end-edges as~$Q_{i_j}$. 
For $j \in  [s+1, s+t]$, by Fact~\ref{fact:absorb}, there exists a compatible path~$Q^*_{i_j}$ on $V(Q_{i_j}) \cup V(P_j)$ with the same end-edges as~$Q_{i_j}$. 
Let $C^*$ be obtained from~$C$ by replacing~$Q_{i_j}$ with~$Q^*_{i_j}$ for each $j \in [s+t]$. 
Note that $C^*$ is a compatible cycle on $V(H) \cup A$. 
\end{proof}

\section{Proof of the main theorems}

We are now ready to derive \cref{thm:main,thm:reg}. For clarity, we prove each result separately, although the strategy is identical: We first apply \cref{lem:absorbing} to obtain an absorber $A$ with the property that if $G'\coloneqq G-A$ contains a spanning compatible linear forest with few components, then $G$ contains a compatible Hamilton cycle. To find such a linear forest, we first use \cref{thm:regular_subgraph} (in the case of \cref{thm:main}) or \cref{lm:regularising} (in the case of \cref{thm:reg}) to obtain a dense spanning regular subgraph $G''\subseteq G'$. Finally, we apply \cref{cor:nearspanninglinearforest} to obtain the desired linear forest in $G''$.

\begin{proof}[Proof of \cref{thm:main}]
    Let $ \rho >0$ and $\delta>  1/2$. 
    Let $\varepsilon \coloneqq \delta - 1/2$. 
    Fix additional constants satisfying
    \[0<1/n_0\ll \beta \ll \alpha , \eps' \ll\eps,\rho.\]

    Let $\mu\coloneqq (1-\rho)(\delta + \sqrt{2\delta - 1})/4\leq (1-\rho)\delta/{2}$ and let $G$ be a graph on $n\geq n_0$ vertices of minimum degree at least $\delta n$. Let $\mathcal{F}$ be a $\mu n$-bounded incompatibility system for $G$. 

    Let $A\subseteq V(G)$ be the absorber obtained by applying \cref{lem:absorbing}.
    Let $G'\coloneqq G-A$ and observe that $G'$ is a graph on $n'\geq (1-\alpha)n$ vertices. Since $G'$ has minimum degree $\delta(G')\geq (\delta - \alpha) n$, setting $ \delta' \coloneqq \delta(G')/n'$, we see that $\delta' \ge \delta - \alpha \ge 1/2+\eps'$. Moreover, $\mathcal{F}$ induces a $\mu'n'$-bounded incompatibility system for $G'$, where $\mu'\coloneqq \mu/(1-\alpha)$. Let $G''$ be a largest regular spanning subgraph of $G'$ and let $d'$ be such that $G''$ is $d'n'$-regular. We claim that $G''$ contains a large compatible linear forest.

     We have by \cref{thm:regular_subgraph} that
     \[d'\geq \frac{\delta' -\eps'+ \sqrt{2\delta' - 1}}{2} 
     \ge \frac{\delta - \alpha -\eps'+ \sqrt{2\delta - 2\alpha - 1}}{2}.\]
    Note that $\sqrt{2\delta - 1} - \sqrt{2\delta - 2\alpha - 1} \le \frac{2\alpha}{\sqrt{2\delta -1}} \le \frac{2\alpha}{\sqrt{\eps}}$%
    \COMMENT{$\sqrt{2\delta - 1} - \sqrt{2\delta - 2\alpha - 1} \le \frac{2\alpha}{\sqrt{2\delta -1}}$ is equivalent to $(2\delta-1)-\sqrt{(2\delta-1)^2-2\alpha(2\delta-1)}\leq 2\alpha$. Since \[(2\delta-1)^2-2\alpha(2\delta-1)\geq (2\delta-1)^2-4\alpha(2\delta-1)+4\alpha^2=((2\delta-1)-2\alpha)^2,\]
    we indeed have
    \[(2\delta-1)-\sqrt{(2\delta-1)^2-2\alpha(2\delta-1)}\leq(2\delta-1)-((2\delta-1)-2\alpha)=2\alpha.\]}
    and thus we see that
    \[d'
    \ge 
    \left(1-\frac{\alpha +\eps' + \tfrac{2\alpha}{\sqrt{\eps}}}{\delta + \sqrt{2\delta - 1}}\right)\frac{\delta + \sqrt{2\delta - 1}}{2}
     \ge \left(1-\frac{\rho}{2}\right)\frac{\delta + \sqrt{2\delta - 1}}{2}.\]
    In particular,
    \[\mu'= \frac{1-\rho}{1-\alpha}\cdot\frac{\delta + \sqrt{2\delta - 1}}{4}\leq \frac{1-\rho}{(1-\alpha)\left(1-\tfrac{\rho}{2}\right)}\frac{d'}{2}\leq \left(1-\frac{\rho}{2}\right)\frac{d'}{2}.\]
    Thus, \cref{cor:nearspanninglinearforest} implies that $G''$ contains a spanning $\mathcal{F}$-compatible linear forest which contains at most $n^{32/33}\leq \beta n$ components. By \cref{lem:absorbing}, $G$ has an $\mathcal{F}$-compatible Hamilton cycle.
\end{proof}

\begin{proof}[Proof of \cref{thm:reg}]
    Let $\eps, \rho >0$. 
    Fix additional constants satisfying
    \[0< 1/n_0\ll \beta \ll \alpha ,\gamma \ll \eps' \ll \eps,\rho.\]
    
    Let $d \ge 1/2+ \eps$ and $\mu\coloneqq (1-\rho)d/2\leq (1-\rho/2)(d-\gamma)/{2}$, let $G$ be a graph on $n\geq n_0$ vertices and suppose that $(d-\gamma)n\leq d_G(v)\leq (d +\gamma)n$ for each $v\in V(G)$. Let $\mathcal{F}$ be a $\mu n$-bounded incompatibility system for $G$. 

    Let $A\subseteq V(G)$ be the absorber obtained by applying \cref{lem:absorbing} with $d-\gamma, \rho/2$, and $\varepsilon'$ playing the roles of $\delta, \rho$, and~$\varepsilon$, respectively.
    Let $G'\coloneqq G-A$ and observe that $G'$ is a graph on $n'\geq (1-\alpha)n$ vertices. Since $G'$ has minimum degree $\delta(G')\geq (d-\gamma - \alpha) n$, setting $ \delta' \coloneqq \delta(G')/n'$, we see that $\delta' \ge d -\gamma - \alpha \ge 1/2+\eps'$. Moreover, $\mathcal{F}$ induces a $\mu'n'$-bounded incompatibility system for $G'$, where $\mu'\coloneqq \mu/(1-\alpha)$. Let $G''$ be a largest regular spanning subgraph of $G'$ and let $d'$ be such that $G''$ is $d'n'$-regular. We claim that $G''$ contains a large compatible linear forest.

    Observe that  $G'$ has 
    maximum degree $\Delta(G) \le (d+\gamma)n \leq (1+2\alpha)(\delta'+2\gamma+\alpha)n' \leq (\delta'+\eps'^2/5)n'$. Thus \cref{lm:regularising} implies that $d'\geq \delta' -\eps' \ge d - \gamma- \alpha  - \eps' \ge (1-3\eps')d$.
    In particular,  
    \[\mu'= \frac{(1-\rho)}{(1-\alpha)}\frac{d}{2}
    \leq \frac{(1-\rho)}{(1-\alpha)(1-3\eps')}\frac{d'}{2}
    \leq \left(1-\frac{\rho}{2}\right)\frac{d'}{2}\] 
    and so we can apply \cref{cor:nearspanninglinearforest} to obtain a spanning compatible linear forest in $G''$ containing at most $n^{32/33}\leq \beta n$ components. By \cref{lem:absorbing}, $G$ has an $\mathcal{F}$-compatible Hamilton cycle.   
\end{proof}

\section{Open Problems}\label{sec:open_problems}

The  most pertinent open questions remain~\cref{q:main} and the following generalisation to $\mu$ as a function of $\delta$:
\begin{question}\label{q:more_general}
    What is the value of the function $\mu(\delta)$ given by Definition~\ref{def:mu}?
\end{question} 
We anticipate that the answer to~\cref{q:more_general} will be either the upper bounds given by~\cref{prop:upper_bound_reg,prop:upper_bound_nonreg} or the lower bound given by~\cref{thm:main}, though we do not make a guess as to which is correct.

\bigskip
\cref{thm:main,thm:reg} only give asymptotic bounds on $\mu$, and only apply to graphs with minimum degree at least $\left( 1/2+\eps \right)n$. One open problem is whether the $\eps$ and $\rho$ in the statements of these theorems can be removed. This is likely to prove difficult: indeed,~\cref{thm:reg} with $\rho = 0$  and $d = 1$ implies the full Bollob\'{a}s-Erd\H{o}s Conjecture.

As mentioned in the introduction, a different generalisation of Dirac's Theorem to edge-coloured graphs has been to seek rainbow Hamilton cycles, where all edges have distinct colours.
This inspires a generalisation of the local concept of incompatibility systems to a global one. A \emph{global incompatibility system}~$\mathcal{F}$ for~$G$ is a family of `forbidden' pairs of  edges of~$G$ (where the pairs are not necessarily intersecting). As before, we say that a subgraph $H$ in~$G$ is \emph{$\mathcal{F}$-compatible} if there does not exist $e,e' \in E(H)$ such that $\{e,e'\} \in \mathcal{F}$. 
We say a global incompatibility system~$\mathcal{F}$ is \emph{$\Delta$-bounded} if for each edge $e$ there are at most $\Delta$ other edges $e'$ such that $\{e,e'\}\in \mathcal{F}$.

For example, consider an edge-coloured graph~$G$.
If $\mathcal{F}_{\textrm{mono}}'$ consists of \emph{all} pairs of edges of the same colour, then $\mathcal{F}_{\textrm{mono}}'$-compatible is equivalent to rainbow.
Coulson and Perarnau~\cite{CoulsonPerarnau} showed that there exists some sufficiently small $\mu>0$ such that if an edge-coloured graph $G$ with $\delta(G) \ge n/2$  has no colour appearing more than $\mu n$ times, then it contains a rainbow Hamilton cycle. This is equivalent to the statement that if $\mathcal{F}_{\textrm{mono}}'$ is $\mu n$-bounded then $G$ has an $\mathcal{F}_{\textrm{mono}}'$-compatible Hamilton cycle.

\begin{question}\label{q:global}
Does there exist a constant $\mu$ such that every $n$-vertex graph~$G$ with $\delta(G) \ge n/2 $ and a $\mu n$-bounded global incompatibility system~$\mathcal{F}$ has an $\mathcal{F}$-compatible Hamilton cycle?
\end{question} 
A positive answer to this question would generalise both Krivelevich, Lee and Sudakov's result on (local) incompatibility systems as well as Coulson and Perarnau's result on rainbow Hamilton cycles.

\section*{Acknowledgements}

The research leading to this paper was initiated at the Extremal and Probabilistic Combinatorics workshop organised by Shoham Letzter, Richard Montgomery, and Katherine Staden, and hosted by the International Centre for Mathematical Studies (ICMS) in Edinburgh.

While working on this project, Natalie Behague was supported by the European Research Council (ERC) under the European Union Horizon 2020 research and innovation programme (grant agreement No. 947978).
Bertille Granet was supported by the Australian Research Council grant DP240101048.

\bibliographystyle{abbrv}
\bibliography{reference}

\end{document}